\documentclass[12pt]{amsart}
\usepackage{amsmath,mathtools,amsthm,amsfonts,bigints,amssymb, mathrsfs, tikz-cd, xcolor}
\usepackage{tabularx}
\usepackage{multirow}
\newtheorem{theorem}{Theorem}[section]
\newtheorem{lemma}[theorem]{Lemma}

\newtheorem{corollary}[theorem]{Corollary}

\theoremstyle{definition}
\newtheorem{definition}[theorem]{Definition}

\newtheorem{remark}[theorem]{Remark}

\newcommand{\what}{\widehat}

\newcommand{\R}{\mathbb R}%
\newcommand{\C}{\mathbb C}%
\newcommand{\N}{\mathbb N}%
\newcommand{\z}{\mathfrak z}%
\newcommand{\mv}{\mathfrak v}%
\newcommand{\J}{\mathscr J}%
\newcommand{\LL}{\mathcal L}%
\numberwithin{equation}{section}
\makeatletter

\renewcommand\subsubsection{\@secnumfont}{\bfseries}%
\renewcommand\subsubsection{\@startsection{subsubsection}{3}
  \z@{.5\linespacing\@plus.7\linespacing}{-.5em}%
  {\normalfont\bfseries}}

  \makeatother
\makeatletter
\@namedef{subjclassname@2020}{%
  \textup{2020} Mathematics Subject Classification}
\makeatother
\begin{document}

\title[Regularity and pointwise convergence]{Regularity and pointwise convergence for dispersive equations with asymptotically concave phase on Damek-Ricci spaces}

\author[Utsav Dewan]{Utsav Dewan}
\address{Stat-Math Unit, Indian Statistical Institute, 203 B. T. Rd., Kolkata 700108, India}
\email{utsav\_r@isical.ac.in\:,\: utsav97dewan@gmail.com}



\begin{abstract}
We study the Carleson's problem on Damek-Ricci spaces $S$ for dispersive equations:
\begin{equation*} 
\begin{cases}
	 i\frac{\partial u}{\partial t} +\Psi(\sqrt{-\LL} )u=0\:,\:  (x,t) \in S \times \mathbb{R} \:,\\
	u(0,\cdot)=f\:,\: \text{ on } S \:,
	\end{cases}
\end{equation*}
where $\LL= \Delta$, the Laplace-Beltrami operator or $\tilde{\Delta}$, the shifted Laplace-Beltrami operator, so that the corresponding phase function $\psi$ satisfies for some $a \in (0,1)$, the large frequency asymptotic:
\begin{equation*} 
\psi(\lambda)=\lambda^a + \mathcal{O}(1)\:,\:\: \lambda \gg 1\:.
\end{equation*}
For almost everywhere pointwise convergence of the solution $u$ to its radial initial data $f$, we obtain the almost sharp regularity threshold $\beta>a/4$. This result is new even for $\mathbb{R}^n$ and in the special case of the fractional Schr\"odinger equations, generalizes classical Euclidean results of Walther.
\end{abstract}

\maketitle
\tableofcontents

\section{Introduction}
One of the most celebrated problems in Euclidean Harmonic analysis is the Carleson's problem: determining the optimal regularity of the initial condition $f$ of the
Schr\"odinger equation given by
\begin{equation} \label{schrodinger}
\begin{cases}
	 i\frac{\partial u}{\partial t} -\Delta_{\R^n} u=0\:,\:\:\:  (x,t) \in \mathbb{R}^n \times \mathbb{R}\:, \\
	u(0,\cdot)=f\:, \text{ on } \mathbb{R}^n \:,
	\end{cases}
\end{equation}
in terms of the index $\beta$ such that for all $f$ belonging to the inhomogeneous Sobolev space $H^\beta(\mathbb{R}^n)$, the solution $u$ of (\ref{schrodinger}) converges pointwise to $f$, 
\begin{equation} \label{pointwise_convergence}
\displaystyle\lim_{t \to 0+} u(x,t)=f(x)\:,\:\:\text{ almost everywhere }.
\end{equation}
This problem was first studied by Carleson \cite{C} who in dimension $1$, obtained the sufficient condition $\beta \ge 1/4$ for the pointwise convergence (\ref{pointwise_convergence}) to hold. In \cite{DK} Dahlberg-Kenig  showed that  $\beta \ge 1/4$ is also necessary. This completely solved the problem in dimension one and subsequently, posed the question in higher dimensions. After five decades of continuous improvement made by several experts in the field, recently in \cite{DGL, DZ} the sufficient condition $\beta>n/2(n+1)$ has been obtained. This bound is sharp except at the endpoint $\beta=n/2(n+1)$, due to a counterexample by Bourgain \cite{Bourgain}.

The Schr\"odinger equation (\ref{schrodinger}) is a special case of the fractional Schr\"odinger equations for $a>0$:
\begin{equation} \label{frac_schrodinger}
\begin{cases}
	 i\frac{\partial u}{\partial t} +\left(-\Delta_{\R^n}\right)^{a/2} u=0\:,\:\:\:  (x,t) \in \mathbb{R}^n \times \mathbb{R}\:, \\
	u(0,\cdot)=f\:, \text{ on } \mathbb{R}^n \:.
	\end{cases}
\end{equation}
For $a>1$, in dimension $1$, Sj\"olin obtained the sharp bound $\beta \ge 1/4$ in \cite{Sjolin}\:. A higher dimensional analogue of this result was obtained by Prestini in \cite{Prestini} where she obtained the sharp bound $\beta \ge 1/4$ for $a>1$ and radial initial data on $\R^n$. Then more results for $a>1$ and radial initial data on $\R^n$, were obtained  by Sj\"olin (see \cite{Sjolin2, Sjolin3, Sjolin4})\:.

In comparison, in the concave phase case, that is when $a \in (0,1)$ in (\ref{frac_schrodinger}), the situation is drastically different. In dimension $1$, the classical result of Cowling \cite{Cowling} yields a sufficient condition $\beta>a/2$. In \cite{Wa1}, Walther improved this bound down to the almost sharp bound $\beta>a/4$. In his following papers \cite{Wa2, Wa3}, Walther investigated the higher dimensional analogue for radial initial data and again obtained the almost sharp bound $\beta>a/4$\:. 
 
Motivated by the above results of Walther, in this article, we consider the Carleson's problem for dispersive equations with concave phase and radial initial data on Damek-Ricci spaces. Damek-Ricci spaces $S$, also known as Harmonic $NA$ groups, are non-unimodular, solvable extensions of Heisenberg type groups $N$, obtained by letting $A=\R^+$ act on $N$ by homogeneous dilations. The rank one Riemannian Symmetric spaces of noncompact type are the prototypical examples of (and in fact accounts for a very small subclass of the more general class of) Damek-Ricci spaces \cite[p. 643]{ADY}. These spaces are non-positively curved, complete, simply-connected Riemannian manifolds and hence have no conjugate points. Thus a function $f: S \to \C$ is said to be radial if, for all $x$ in $S$, $f(x)$ depends only on the geodesic distance of $x$ from the identity $e$. 

Let $\Delta$ be the Laplace-Beltrami operator
on $S$ corresponding to the left-invariant Riemannian metric. Its $L^2$-spectrum is the half line $(-\infty,  -Q^2/4]$, where $Q$ is the homogeneous dimension of $N$. The corresponding shifted Laplace-Beltrami operator is given by, $\tilde{\Delta}:=\Delta + \frac{Q^2}{4}$\:. For $\LL= \Delta$ and $\tilde{\Delta}$, we consider the dispersive equations
\begin{equation} \label{dispersive}
\begin{cases}
	 i\frac{\partial u}{\partial t} +\Psi(\sqrt{-\LL} )u=0\:,\:  (x,t) \in S \times \R \:,\\
	u(0,\cdot)=f\:,\: \text{ on } S \:,
	\end{cases}
\end{equation}
where $\Psi: [0,\infty) \to \R$ satisfies suitable conditions. Then for a radial function $f$ belonging to the $L^2$-Schwartz class $\mathscr{S}^2(S)_o$ (for the definition see (\ref{schwartz_defn})), the solution of (\ref{dispersive}) is given by,
\begin{equation} \label{dispersive_soln}
S_{\psi,t} f(x):= \int_{0}^\infty \varphi_\lambda(x)\:e^{it\psi(\lambda)}\:\widehat{f}(\lambda)\: {|{\bf c}(\lambda)|}^{-2}\: d\lambda\:,
\end{equation}
where $\varphi_\lambda$ is the spherical function, $\psi$ is the phase function of the corresponding multiplier, $\widehat{f}$ is the Spherical Fourier transform of $f$ and ${\bf c}(\cdot)$ denotes the Harish-Chandra's ${\bf c}$-function. The regularity of the initial data, on the other hand, is realized by considering the fractional $L^2$-Sobolev spaces on $S$. For $\beta \ge 0$, the above Sobolev spaces specialized to radial  functions are defined as \cite{APV}:
\begin{eqnarray} \label{sobolev_space_defn}
&& H^\beta(S)\\
&:=&\left\{f \in L^2(S): {\|f\|}_{H^\beta(S)}:= {\left(\int_0^\infty {\left(\lambda^2 + \frac{Q^2}{4}\right)}^\beta {|\widehat{f}(\lambda)|}^2 {|{\bf c}(\lambda)|}^{-2} d\lambda\right)}^{1/2}< \infty\right\}\nonumber.
\end{eqnarray}

In the recent works \cite{Dewan, DR, Dewan2}, we have studied the Carleson's problem for some well-known dispersive equations with convex phase such as fractional Schr\"odinger equations with $a>1$, the Beam equation and the Boussinesq equation, with radial initial data on $S$. To complement the above results, in this article, we consider dispersive equations with phase satisfying the following notion of {\em asymptotic concavity}:
\begin{definition} \label{asymp_concavity}
For $a \in (0,1)$, a dispersive equation (\ref{dispersive}) is called {\em asymptotically concave of degree $a$}, if the phase function $\psi$ of its multiplier is a real-valued continuous function on $[0,\infty)$, $C^\infty$ away from origin and satisfies,
\begin{equation} \label{concavity}
\psi(\lambda)=\lambda^a + \mathcal{O}(1)\:,\:\: \lambda \gg 1\:.
\end{equation}
\end{definition}

\begin{remark} \label{example}
For $a \in (0,1)$, the particular choice $\Psi(\lambda):=\lambda^a$ in (\ref{dispersive}) yields the fractional Schr\"odinger equations corresponding to $\Delta$ and $\tilde{\Delta}$\:. Both of these equations are {\em asymptotically concave of degree $a$} as the corresponding phases are given by,
\begin{equation} \label{example_phase}
\psi_1(\lambda)= \left(\lambda^2+\frac{Q^2}{4}\right)^{\frac{a}{2}}\:\text{ and } \psi_2(\lambda)= \lambda^a \:\:\text{ respectively }.
\end{equation}  
\end{remark}

To study the pointwise convergence of the solution to its initial data, we define the associated maximal function,
\begin{equation} \label{maximal_fn_defn}
S_\psi^* f(x):= \displaystyle\sup_{0<t<1} \left|S_tf(x)\right|\:.
\end{equation}
Our first result is the following $L^2_{loc}$ boundedness for the maximal function:
\begin{theorem} \label{maximal_bddness_thm}
Let $f$ be a radial $L^2$-Schwartz class function on $S$ and let $B_R$ denote the geodesic ball centered at the identity with radius $R>0$. Then for asymptotically concave dispersive equations of degree $a \in (0,1)$, 
\begin{equation} \label{maximal_bddness_inequality}
{\|S_\psi^* f\|}_{L^2\left(B_R\right)} \lesssim \: {\|f\|}_{H^\beta(S)}\:,
\end{equation}
holds for all $R>0$ and $\beta > a/4$, with the implicit constant depending only on the radius $R$, the parameters $a$, $\beta$ and the intrinsic geometry of the space.
\end{theorem}

Then by standard arguments in the literature (for instance see the proof of Theorem 5 of \cite{Sjolin}), Theorem \ref{maximal_bddness_thm} yields the positive result on pointwise convergence:
\begin{corollary} \label{pointwise_conv_cor}
The solution $u$ of an asymptotically concave dispersive equation of degree $a \in (0,1)$, converges pointwise to the radial initial data $f$,
\begin{equation*}
\displaystyle\lim_{t \to 0+} u(x,t)=f(x)\:,
\end{equation*}
for almost every $x$ in $S$ with respect to the left Haar measure on $S$, whenever $f \in H^\beta(S)$ with $\beta > a/4$.
\end{corollary}   
 
Our next result illustrates the almost sharpness of the above results:
\begin{theorem} \label{sharpness_thm}
For asymptotically concave dispersive equations of degree $a \in (0,1)$, the maximal estimate (\ref{maximal_bddness_inequality}) fails if $\beta<a/4$\:.
\end{theorem}

The key idea in the proofs of Theorems \ref{maximal_bddness_thm} and \ref{sharpness_thm}, is to estimate the joint oscillation afforded by the spherical function and the corresponding multiplier and is motivated by the works of Walther \cite{Wa1, Wa2, Wa3}. But there are two major technical departures:
\begin{itemize}
\item Compared to the case of Walther, our results deal with more general phases which are given only in terms of the asymptotic condition (\ref{concavity}). Our approach addresses this general scenario by means of a transference principle (see Lemma \ref{transference_principle}). 
\item Unlike in the Euclidean case, the spherical function on a Damek-Ricci space is not merely a single Bessel function. In fact, it only admits certain series expansions (the Bessel series expansion and a series expansion similar to the classical Harish-Chandra series expansion) depending on the geodesic distance from the identity of the group. Then to capture the aforementioned oscillation from these expansions, to decompose the linearized maximal function into suitable pieces and then to estimate each of them individually, become somewhat technical. In this regard, we also need the improved estimates obtained by Anker-Pierfelice-Vallarino on the coefficients of the above series expansion of the spherical functions (see (\ref{coefficient_estimate})).
\end{itemize}

In this article, we include the case when $N$ is abelian and $\mathfrak{n}$ coincides with its center in the definition of $H$-type groups, as a degenerate case, so that the Real hyperbolic spaces are also included in the class of Damek-Ricci spaces (see \cite[pp. 209-210]{CDKR}). 

This article is organized as follows. In section $2$, we recall certain aspects of Euclidean Fourier Analysis, the essential preliminaries about Damek-Ricci spaces and Spherical Fourier Analysis thereon and a useful oscillatory integral estimate. In section $3$, we obtain some auxiliary results that are required in the proof of our main results and may be of independent interest. Theorems \ref{maximal_bddness_thm} and \ref{sharpness_thm} are proved in sections $4$ and $5$ respectively. Finally, we conclude in section $6$, by making some remarks and posing some new problems.
 
\section{Preliminaries}
In this section, we recall some preliminaries and fix our notations.
\subsection{Some notations}
Throughout, the symbols `c' and `C' will denote positive constants whose values may change on each occurrence. The enumerated constants $c_0,c_1, \dots$ will however be fixed throughout. $\N$ will denote the set of positive integers. We will also require the notation,
\begin{equation*}
\frac{1}{2}\N =\left\{\frac{1}{2}n \mid n \in \N\right\}\:.
\end{equation*}
For non-negative functions $f_1,\:f_2,\:f_3$ we write, $f_1 \lesssim f_2$ (resp. $f_1 \gtrsim f_2$) if there exists a constant $C \ge 1$, so that
\begin{equation*}
f_1 \le C f_2 \:,\text{ (resp. } f_1 \ge \frac{1}{C} f_2 \text{)}\:.
\end{equation*}
$f_1 \asymp f_2$ will mean there exist constants $C,C'>0$, so that
\begin{equation*}
C f_1 \le f_2 \le C' f_1\:.
\end{equation*}
We write, $f_1=f_2+\mathcal{O}(f_3)$ if 
\begin{equation*}
|f_1-f_2| \lesssim f_3\:.
\end{equation*}
By $\lambda \gg 1$, we mean $\lambda \ge M$ for some sufficiently large $M>1$\:. 
 
\subsection{Fourier Analysis on $\R^n$:}
In this subsection, we recall some Euclidean Fourier Analysis, most of which can be found in \cite{SW}. On $\R$, for ``nice" functions $f$, the Fourier transform $\tilde{f}$ is defined as
\begin{equation*}
\tilde{f}(\xi)= \int_\R f(x)\: e^{-ix\xi}\:dx\:.
\end{equation*}
An important inequality in one-dimensional Fourier Analysis is the Pitt's inequality \cite[p. 489]{Stein}:
\begin{equation} \label{Pitt's_ineq}
{\left(\int_\R {\left|\tilde{f}(\xi)\right|}^2 \:{|\xi|}^{-2 \alpha} d\xi\right)}^{1/2} \lesssim {\left(\int_\R {\left|f(x)\right|}^p\: {|x|}^{\alpha_1 p} dx\right)}^{1/p}\:,
\end{equation}
where $\alpha_1 = \alpha + \frac{1}{2}- \frac{1}{p}$ and the following two conditions are satisfied:
\begin{equation} \label{Pitt's_ineq_cond}
0 \le \alpha_1 < 1 - \frac{1}{p}\:,\: \text{ and } 0 \le \alpha < \frac{1}{2}\:.
\end{equation}
A $C^\infty$ function $f$ on $\R$ is called a Schwartz class function if 
\begin{equation*}
\left|{\left(\frac{d}{dx}\right)}^M f(x)\right| \lesssim {(1+|x|)}^{-N} \:, \text{ for any } M, N \in \N \cup \{0\}\:.
\end{equation*}
We denote by $\mathscr{S}(\R)$ the class of all such functions and $\mathscr{S}(\R)_{e}$ will denote the even subclass of $\mathscr{S}(\R)$. Similarly, $C^\infty_c(\R)_e$ denotes the collection of all even, compactly supported smooth functions on $\R$. 

Let $\beta \in \C$ with $Re(\beta)>0$. The Riesz potential of order $\beta$ is the operator
\begin{equation*}
I_\beta= {(-\Delta_{\R})}^{-\beta/2}\:,
\end{equation*}
which can also be written as,
\begin{equation*}
I_\beta(f)(x)=C \int_\R f(y)\: {|x-y|}^{\beta -1}\:dy\:,
\end{equation*}
for some $C>0$ (depending on $\beta$), whenever $f \in \mathscr{S}(\R)$. For $\beta >0$, the Fourier transform of the Riesz potential of $f \in \mathscr{S}(\R)$ satisfies the following identity:
\begin{equation} \label{riesz_identity}
(I_\beta(f))^\sim(\xi)= C_\beta {|\xi|}^{-\beta}\tilde{f}(\xi)\:.
\end{equation}
For a ``nice" radial function $f$ in $\R^n$, $n \ge 2$, the Fourier transform is defined as
\begin{equation*}
\mathscr{F}f(\lambda):= \int_0^\infty f(r) \J_{\frac{n-2}{2}}(\lambda r)\: r^{n-1}\: dr\:,\:\:\:\:\:\lambda\in [0,\infty),
\end{equation*}
where for all $\mu \ge 0\:, \J_\mu$ denotes the modified Bessel function:
\begin{equation*}
\J_\mu(z)= 2^\mu \: \pi^{1/2} \: \Gamma\left(\mu + \frac{1}{2} \right) \frac{J_\mu(z)}{z^\mu}\:.
\end{equation*}
Here $J_\mu$ are the Bessel functions \cite[p. 154]{SW}. The asymptotic oscillation of Bessel functions is well-studied:
\begin{lemma}\cite[lemma 1]{Prestini} \label{bessel_function_expansion}
Let $\mu \in \frac{1}{2}\N$. Then
there exists a positive constant $A_\mu$ such that
\begin{equation*}
J_\mu(s)= \sqrt{\frac{2}{\pi s}} \cos \left( s - \frac{\pi}{2}\mu - \frac{\pi}{4}\right) + \tilde{E}_\mu(s)\:,
\end{equation*}
where
\begin{equation*}
|\tilde{E}_\mu(s)| \lesssim s^{-3/2}\:,\:\text{ for } s \ge A_\mu\:.
\end{equation*}
\end{lemma}

\subsection{Damek-Ricci spaces and spherical Fourier Analysis thereon:}
In this section, we will explain the notations and state relevant results on Damek-Ricci spaces. Most of these results can be found in \cite{ADY, APV}.

Let $\mathfrak n$ be a two-step real nilpotent Lie algebra equipped with an inner product $\langle, \rangle$. Let $\mathfrak{z}$ be the center of $\mathfrak n$ and $\mathfrak v$ its orthogonal complement. We say that $\mathfrak n$ is an $H$-type algebra if for every $Z\in \mathfrak z$ the map $J_Z: \mathfrak v \to \mathfrak v$ defined by
\begin{equation*}
\langle J_z X, Y \rangle = \langle [X, Y], Z \rangle, \:\:\:\: X, Y \in \mathfrak v
\end{equation*}
satisfies the condition $J_Z^2 = -|Z|^2I_{\mathfrak v}$, $I_{\mathfrak v}$ being the identity operator on $\mathfrak v$. A connected and simply connected Lie group $N$ is called an $H$-type group if its Lie algebra is $H$-type. Since $\mathfrak n$ is nilpotent, the exponential map is a diffeomorphism
and hence we can parametrize the elements in $N = \exp \mathfrak n$ by $(X, Z)$, for $X\in \mathfrak v, Z\in \mathfrak z$. The group $A = \R^+$ acts on an $H$-type group $N$ by nonisotropic dilation: $(X, Z) \mapsto (\sqrt{a}X, aZ)$. Let $S = NA$ be the semidirect product of $N$ and $A$ under the above action. Then $S$ is a solvable, connected and simply connected Lie group having Lie algebra $\mathfrak s = \mathfrak v \oplus \mathfrak z \oplus \R$\:. We note
that for any $Z \in \mathfrak z$ with $|Z| = 1$, $J_Z^2 = -I_{\mathfrak v}$; that is, $J_Z$ defines a complex structure
on $\mathfrak v$ and hence $\mathfrak v$ is even dimensional. $m_\mv$ and $m_z$ will denote the dimension of $\mv$ and $\z$ respectively. Let $n$ and $Q$ denote dimension and the homogenous dimension of $S$ respectively:
\begin{equation*}
n=m_{\mv}+m_\z+1, \:\:\:\:\:\:\: Q = \frac{m_\mv}{2} + m_\z.
\end{equation*}

The group $S$ is equipped with the left-invariant Riemannian metric induced by
\begin{equation*}
\langle (X,Z,l), (X',Z',l') \rangle = \langle X, X' \rangle + \langle Z, Z' \rangle + ll'
\end{equation*}
on $\mathfrak s$. For $x \in S$, we denote by $s=d(e,x)$, that is, the geodesic distance of $x$ from the identity $e$. Then the left Haar measure $dx$ of the group $S$ may be normalized so that
\begin{equation*}
dx= A(s)\:ds\:d\sigma(\omega)\:,
\end{equation*}
where $A$ is the density function given by,
\begin{equation*}
A(s)= 2^{m_\mv + m_\z} \:{\left(\sinh (s/2)\right)}^{m_\mv + m_\z}\: {\left(\cosh (s/2)\right)}^{m_\z} \:,
\end{equation*}
and $d\sigma$ is the surface measure of the unit sphere. By elementary estimates of hyperbolic functions, we have the asymptotics:
\begin{equation} \label{density_function}
A(s) \asymp s^{n-1}\:,\:\:\:\:\:\: 0<s\le R < \infty\:.
\end{equation}

For a radial function $f$, we then have
\begin{equation*}
\int_S f(x)~dx=\int_{0}^\infty f(s)~A(s)~ds\:.
\end{equation*}

We now recall the spherical functions on Damek-Ricci spaces. The spherical functions $\varphi_\lambda$ on $S$, for $\lambda \in \C$ are the radial eigenfunctions of the Laplace-Beltrami operator $\Delta$, satisfying the following normalization criterion
\begin{equation*}
\begin{cases}
 & \Delta \varphi_\lambda = - \left(\lambda^2 + \frac{Q^2}{4}\right) \varphi_\lambda  \\
& \varphi_\lambda(e)=1 \:.
\end{cases}
\end{equation*}
For all $\lambda \in \R$ and $x \in S$, the spherical functions satisfy
\begin{equation*}
\varphi_\lambda(x)=\varphi_\lambda(s)= \varphi_{-\lambda}(s)\:.
\end{equation*}
It also satisfies for all $\lambda \in \R$ and all $s \ge 0$:
\begin{equation} \label{phi_lambda_bound}
\left|\varphi_\lambda(s)\right| \le 1\:.
\end{equation}

The spherical function is instrumental in defining the Spherical Fourier transform of a ``nice" radial function $f$ (on $S$):
\begin{equation*}
\widehat{f}(\lambda):= \int_S f(x) \varphi_\lambda(x) dx = \int_0^\infty f(s) \varphi_\lambda(s) A(s) ds\:.
\end{equation*}
The Harish-Chandra ${\bf c}$-function is defined as
\begin{equation*}
{\bf c}(\lambda)= \frac{2^{(Q-2i\lambda)} \Gamma(2i\lambda)}{\Gamma\left(\frac{Q+2i\lambda}{2}\right)} \frac{\Gamma\left(\frac{n}{2}\right)}{\Gamma\left(\frac{m_\mv + 4i\lambda+2}{4}\right)}\:,
\end{equation*}
for all $\lambda \in \R$. We will need the following pointwise estimates (see \cite[Lemma 4.8]{RS}):
\begin{equation} \label{plancherel_measure}
{|{\bf c}(\lambda)|}^{-2} \asymp \:{|\lambda|}^2 {\left(1+|\lambda|\right)}^{n-3}\:.
\end{equation}
For $j \in \N \cup \{0\}$, we also have the
derivative estimates (\cite[Lemma 4.2]{A}):
\begin{equation}\label{c-fn_derivative_estimates}
\left|\frac{d^j}{d \lambda^j}{|{\bf c}(\lambda)|}^{-2}\right| \lesssim_j {(1+|\lambda|)}^{n-1-j} \:, \:\: \lambda \in \R.
\end{equation}

One has the following inversion formula (when valid) for radial functions:
\begin{equation*}
f(x)= C \int_{0}^\infty \widehat{f}(\lambda)\varphi_\lambda(x) {|{\bf c}(\lambda)|}^{-2}
d\lambda\:,
\end{equation*}
where $C$ depends only on $m_\mv$ and $m_\z$. Moreover, the Spherical Fourier transform extends to an isometry from the space of radial $L^2$ functions on $S$ onto $L^2\left((0,\infty),C{|{\bf c}(\lambda)|}^{-2} d\lambda\right)$. 

Similar to the inhomogeneous Sobolev spaces (defined in (\ref{sobolev_space_defn})), one also has the notion of homogeneous Sobolev spaces. For $\beta \ge 0$, the homogeneous Sobolev spaces specialized for radial  functions (corresponding to the shifted Laplace-Beltrami operator $\tilde{\Delta})$ are defined as,
\begin{equation*}
\dot{H}^\beta(S):=\left\{f \in L^2(S): {\|f\|}_{\dot{H}^\beta(S)}:= {\left(\int_0^\infty \lambda^{2\beta}\: {|\widehat{f}(\lambda)|}^2 {|{\bf c}(\lambda)|}^{-2} d\lambda\right)}^{1/2}< \infty\right\}\:.
\end{equation*}
Clearly for any $\beta \ge 0$ and $f \in H^\beta(S)$,
\begin{equation} \label{sobolev_inclusion}
 {\|f\|}_{\dot{H}^\beta(S)} \le  {\|f\|}_{H^\beta(S)}\:\text{ and hence } H^\beta(S) \hookrightarrow \dot{H}^\beta(S)\:.
\end{equation}

The class of radial $L^2$-Schwartz class functions on $S$, denoted by $\mathscr{S}^2(S)_{o}$, is defined to be the collection of $f \in C^\infty(S)_{o}$ such that 
\begin{equation} \label{schwartz_defn}
\left|{\left(\frac{d}{ds}\right)}^M f(s)\right| \lesssim {(1+s)}^{-N} e^{-\frac{Q}{2}s}\:, \text{ for any } M, N \in \N \cup \{0\}\:,
\end{equation}
where $C^\infty(S)_{o}$ is the set of radial smooth functions on $S$ (see \cite[p. 652]{ADY}).

One important tool which relates the spherical Fourier transform on $S$ with the Fourier transform on $\R$ is the so called Abel transform (see \cite{cowl, RS}). Indeed,
we have the following commutative diagram, where every map is a topological isomorphism:
\[
\begin{tikzcd}[row sep=1.4cm,column sep=1.4cm]
\mathscr{S}^2(S)_{o}\arrow[r,"\mathscr{A}_{S,\R}"] \arrow[dr,swap,"\wedge"] & {\mathscr{S}(\R)}_{e}
\arrow[d,"\sim"] \\
&  {\mathscr{S}(\R)}_{e}&
\end{tikzcd}
\]

\begin{itemize}
\item $\mathscr{A}_{S,\R}$ is the Abel transform defined from $\mathscr{S}^2(S)_{o}$ to ${\mathscr{S}(\R)}_{e}$\:.
\item $\wedge$ denotes the Spherical Fourier transform from $\mathscr{S}^2(S)_{o}$ to ${\mathscr{S}(\R)}_{e}$\:.
\item $\sim$ denotes the 1-dimensional Euclidean Fourier transform from ${\mathscr{S}(\R)}_{e}$ to itself.
\end{itemize}

For more details regarding the Abel transform, we refer the reader to \cite[p. 652-653]{ADY}. These have been generalized to the setting of Ch\'ebli-Trim\`eche Hypergroups \cite{BX}, whose simplest case is that of radial functions on $\R^n$.

We now define the following normalizing constant in terms of the Gamma functions,
\begin{equation*}
c_0 = 2^{m_\z}\: \pi^{-1/2}\: \frac{\Gamma(n/2)}{\Gamma((n-1)/2)}\:.
\end{equation*}
For points near the identity, one has the following Bessel series expansion of $\varphi_\lambda$:
\begin{lemma}\cite[Theorem 3.1]{A} \label{bessel_series_expansion}
There exist $R_0, 2<R_0<2R_1$, such that for any $0 \le s \le R_0$, and any integer $M \ge 0$, and all $\lambda \in \R$, we have
\begin{equation*}
\varphi_\lambda(s)= c_0 {\left(\frac{s^{n-1}}{A(s)}\right)}^{1/2} \displaystyle\sum_{l=0}^M a_l(s)\J_{\frac{n-2}{2}+l}(\lambda s) s^{2l} + E_{M+1}(\lambda,s)\:,
\end{equation*}
where
\begin{equation*}
a_0 \equiv 1\:,\: |a_l(s)| \le C {(4R_1)}^{-l}\:,
\end{equation*}
and the error term has the following behaviour
\begin{equation*}
	\left|E_{M+1}(\lambda,s) \right| \le C_M \begin{cases}
	 s^{2(M+1)}  & \text{ if  }\: |\lambda s| \le 1 \\
	s^{2(M+1)} {|\lambda s|}^{-\left(\frac{n-1}{2} + M +1\right)} &\text{ if  }\: |\lambda s| > 1 \:.
	\end{cases}
\end{equation*}
Moreover, for every $0 \le s <2$, the series
\begin{equation*}
\varphi_\lambda(s)= c_0 {\left(\frac{s^{n-1}}{A(s)}\right)}^{1/2} \displaystyle\sum_{l=0}^\infty a_l(s)\J_{\frac{n-2}{2}+l}(\lambda s) s^{2l}\:,
\end{equation*}
is absolutely convergent.
\end{lemma}

For the asymptotic behaviour of the spherical functions away from the identity, we look at the following series expansion \cite[pp. 735-736]{APV}:
\begin{equation} \label{anker_series_expansion}
\varphi_\lambda(s)= 2^{-m_\z/2} {A(s)}^{-1/2} \left\{{\bf c}(\lambda)  \displaystyle\sum_{\mu=0}^\infty \Gamma_\mu(\lambda) e^{(i \lambda-\mu) s} + {\bf c}(-\lambda) \displaystyle\sum_{\mu=0}^\infty \Gamma_\mu(-\lambda) e^{-(i\lambda + \mu) s}\right\}\:.
\end{equation}
The above series converges for $\lambda \in \R \setminus \{0\}$, uniformly on compacts not containing the group identity, where $\Gamma_0 \equiv 1$ and for $\mu \in \N$, one has the recursion formula,
\begin{equation*}
(\mu^2-2i\mu\lambda) \Gamma_\mu(\lambda) = \displaystyle\sum_{j=0}^{\mu -1}\omega_{\mu -j}\Gamma_{j}(\lambda)\:.
\end{equation*}
Then one has the following estimate on the coefficients \cite[Lemma 1]{APV}, for constants $C>0, d \ge 0$:
\begin{equation} \label{coefficient_estimate}
\left|\Gamma_\mu(\lambda)\right| \le C \mu^d {\left(1+|\lambda|\right)}^{-1}\:,
\end{equation}
for all $\lambda \in \R \setminus \{0\}, \mu \in \N$. 

The relevant preliminaries on series expansions of $\varphi_\lambda$, for the degenerate case of the Real hyperbolic spaces  can be found in \cite{ST, AP}. 

\subsection{An oscillatory integral estimate:} In this subsection, we state an oscillatory integral estimate which is the heart of the matter. For $a \in (0,1)$ and $\beta \in \left(\frac{a}{4},\min\left\{\frac{a}{2},\frac{1}{4}\right\}\right)$, we set
\begin{equation} \label{Walther_constants}
c_1=1-2\beta\:\:\:\text{ and } c_2= \frac{4\beta-2+a}{2a-2}\:.
\end{equation}
It is then easy to note that
\begin{equation} \label{Walther_constants_relations}
0<c_1 < c_2<1\:.
\end{equation}
We now present the oscillatory integral estimate:
\begin{lemma}\cite[Lemma 4.6, p. 196]{Wa2} \label{oscillatory_integral_estimate}
Let $a \in (0,1)\:,\:\beta \in \left(\frac{a}{4},\min\left\{\frac{a}{2},\frac{1}{4}\right\}\right)$ and an even $\chi \in C^\infty_c(\R)$ be $[0,1]$-valued such that 
\begin{eqnarray*}
\chi(\lambda)&=& 1,\:\:\:\:|\lambda| < 1,\\
&=&0,\:\:\:\:|\lambda| \ge 2\:.
\end{eqnarray*}
Then for any $A>0$, there is a positive number $C_A$ independent of $\varepsilon \in [-A,A]$, $N \in \N$ and $x \in \R$ such that
\begin{equation*}
\left|\int_{\R} e^{i\left(x \xi + \varepsilon |\xi|^a\right)}\:|\xi|^{-2\beta}\:\chi\left(\frac{\xi}{N}\right)\:d\xi\right| \le C_A \left(|x|^{-c_1}+ |x|^{-c_2}\right)\:.
\end{equation*}  
\end{lemma}

We remark that Lemma $4.6$ in \cite{Wa2} is stated for $\varepsilon \in [0,A]$, but its proof works for $\varepsilon \in [-A,0]$ as well.

\section{Some auxiliary lemmata}
In this section, we obtain some auxiliary results that are required in the proof of our main results and may be of independent interest.  
\subsection{A Schwartz correspondence for high frequency:}
In this subsection, we obtain a one-one correspondence between radial Schwartz class functions on $S$ and $\R^n$ whose Spherical Fourier transforms are supported away from the origin.

We start off by defining 
\begin{equation*}
\mathscr{S}(\R)^\infty_{e}:=\{g\in \mathscr{S}(\R)_{e}\mid 0\notin Supp({g})\}.
\end{equation*}
So, $\mathscr{S}(\R)^\infty_{e}$ is the collection of all even Schwartz class functions on $\R$ which are supported outside an interval containing $0$. Then we look at its images under the Euclidean and Spherical Fourier inversions:
\begin{equation} \label{new_schwartz_spaces}
\mathscr{S}(\R^n)^\infty_{o} := \mathscr{F}^{-1}\left(\mathscr{S}(\R)^\infty_{e}\right)\:, \:\:\text{ and } \mathscr{S}^2(S)^\infty_{o} := \wedge^{-1}\left(\mathscr{S}(\R)^\infty_{e}\right)\:. 
\end{equation}
We now present the following one-one correspondence:
\begin{lemma} \label{schwartz_correspondence}
For each $f \in \mathscr{S}^2(S)^\infty_{o}$, there exists a unique $g \in \mathscr{S}(\R^n)^\infty_{o}$ (and conversely) such that 
\begin{equation*}
Supp(\mathscr{F}g)=Supp(\widehat{f})\:,
\end{equation*}
with
\begin{equation*}
\lambda^{n-1} \mathscr{F}g(\lambda) = {|{\bf c}(\lambda)|}^{-2} \widehat{f}(\lambda)\:,\:\: \lambda \in (0,\infty)\:.
\end{equation*}
\end{lemma}
\begin{proof}
By the properties of the Abel transform described before, we have the following commutative diagram,
\[
\begin{tikzcd}[row sep=1.4cm,column sep=1.4cm]
\mathscr{S}^2(S)_{o}\arrow[r,"\mathscr{A}_{S,\R}"] \arrow[dr,swap,"\wedge"] & {\mathscr{S}(\R)}_{e}
\arrow[d,"\sim"] &
\arrow[l,"\mathscr{A}_{\R^n,\R}",swap]  \mathscr{S}(\R^n)_{o} \arrow[dl,"\mathscr{F}"] \\
&  {\mathscr{S}(\R)}_{e}&
\end{tikzcd}
\]
where $\mathscr{A}_{\R^n,\R}$ is the Abel transform defined from $\mathscr{S}(\R^n)_{o}$ to ${\mathscr{S}(\R)}_{e}$. As all the maps above are bijections, by defining
\begin{equation*}
\mathscr{A}:=\mathscr{A}^{-1}_{\R^n,\R} \circ \mathscr{A}_{S,\R},
\end{equation*}
we reduce matters to the following simplified commutative diagram:

\[
\begin{tikzcd}[row sep=1.4cm,column sep=1.4cm]
\mathscr{S}^2(S)_{o}\arrow[r,"\mathscr{A}"] \arrow[dr,swap,"\wedge"] & \mathscr{S}(\R^n)_{o}
\arrow[d,"\mathscr{F}"]   \\
&  {\mathscr{S}(\R)}_{e}&
\end{tikzcd}
\]

Then by the definition of the spaces $\mathscr{S}(\R)^\infty_{e},\:\mathscr{S}(\R^n)^\infty_{o}$ and $\mathscr{S}^2(S)^\infty_{o}$, the above commutative diagram descends to the following, where every arrow is again a bijection:

\[
\begin{tikzcd}[row sep=1.4cm,column sep=1.4cm]
\mathscr{S}^2(S)^\infty_{o}\arrow[r,"\mathscr{A}"] \arrow[dr,swap,"\wedge"] & \mathscr{S}(\R^n)^\infty_{o}
\arrow[d,"\mathscr{F}"]   \\
&  {\mathscr{S}(\R)}^\infty_{e}&
\end{tikzcd}
\]

Consequently, for $f \in \mathscr{S}^2(S)^\infty_{o}$, there exists a unique $\mathscr{A}f \in \mathscr{S}(\R^n)^\infty_{o}$ (and conversely) such that
\begin{equation} \label{ball_pf_eq3}
\widehat{f}(\lambda)= \mathscr{F}(\mathscr{A}f)(\lambda)\:.
\end{equation}
Next let us define a map $\mathfrak{m}: \mathscr{S}(\R)^\infty_{e}\to \mathscr{S}(\R)^\infty_{e}$, by the formula
\begin{equation*}
\mathfrak{m}(\kappa)(\lambda):= \frac{{|{\bf c}(\lambda)|}^{-2}}{\lambda^{n-1}} \kappa(\lambda).
\end{equation*}
Because of the derivative estimates of ${|{\bf c}(\lambda)|}^{-2}$ given by (\ref{c-fn_derivative_estimates}), the map $\mathfrak m$ is well defined and is a bijection with the inverse given by,
\begin{equation*}
\mathfrak{m}^{-1}(\kappa)(\lambda):= \frac{\lambda^{n-1}}{{|{\bf c}(\lambda)|}^{-2}} \kappa(\lambda).
\end{equation*}
As the Euclidean Fourier transform $\mathscr{F}$ is a bijection between $\mathscr{S}(\R^n)^\infty_{o}$ and $\mathscr{S}(\R)^\infty_{e}$, we get an induced map $\mathcal{M}$ obtained by conjugating $\mathfrak{m}$ with $\mathscr{F}$, which is a bijection from $\mathscr{S}(\R^n)^\infty_{o}$ onto itself:

\[
\begin{tikzcd}[row sep=1.4cm,column sep=1.4cm]
\mathscr{S}(\R^n)^\infty_{o} \arrow[r,"\mathcal{M}"] \arrow[d,swap,"\mathscr{F}"] & \mathscr{S}(\R^n)^\infty_{o}
\arrow[d,"\mathscr{F}"]   \\
{\mathscr{S}(\R)}^\infty_{e}  \arrow[r,"\mathfrak{m}"]  &  {\mathscr{S}(\R)}^\infty_{e}&
\end{tikzcd}
\]

Thus by (\ref{ball_pf_eq3}), we get the unique choice $\mathcal{M}(\mathscr{A}f) \in \mathscr{S}(\R^n)^\infty_{o}$ such that
\begin{equation} \label{ball_pf_eq4}
\mathscr{F}(\mathcal{M}(\mathscr{A}f))(\lambda)= \frac{{|{\bf c}(\lambda)|}^{-2}}{\lambda^{n-1}}  \mathscr{F}(\mathscr{A}f)(\lambda)= \frac{{|{\bf c}(\lambda)|}^{-2}}{\lambda^{n-1}} \widehat{f}(\lambda)\:.
\end{equation}
We are thus through by defining $g:= \mathcal{M}(\mathscr{A}f)$\:.
\end{proof}
\subsection{A transference principle for dispersive equations:} In this subsection, we observe that from the viewpoint of obtaining $L^2_{loc}$-maximal estimates, the fractional Schr\"odinger equation of degree $a \in (0,1)$, corresponding to $\tilde{\Delta}$  is prototypical of any asymptotically concave dispersive equation of degree $a \in (0,1)$. This follows from an abstract local transference principle, obtained by the author in \cite{Dewan2}\:. For the sake of completeness, we briefly mention it here.

\begin{definition} \cite[Definition $1.4$]{Dewan2}\label{my_defn}
Let us consider two dispersive equations of the form (\ref{dispersive}) corresponding to $\Psi_1,\Psi_2$ (recall that $\psi_1,\psi_2$ are the phase functions of the corresponding multipliers).
\begin{itemize}
\item[(i)] They are called $(\psi_1,\psi_2)$-{\it locally transferrable} if given that for all $R>0$, some $\beta_0>0$ and some $p \in [1,\infty]$, the maximal estimate 
\begin{equation*}
{\|S^*_{\psi_1} f\|}_{L^p(B_R)} \lesssim {\|f\|}_{H^\beta(S)}\:,
\end{equation*}
holds for all $\beta >\beta_0$ and all $f \in \mathscr{S}^2(S)_o$, it follows that the maximal estimate 
\begin{equation*}
{\|S^*_{\psi_2} f\|}_{L^p(B_R)} \lesssim {\|f\|}_{H^\beta(S)}\:,
\end{equation*}
also holds for all $R>0$, all $\beta >\beta_0$ and all $f \in \mathscr{S}^2(S)_o$. 

\item[(ii)] If the two equations are both $(\psi_1,\psi_2)$-locally transferrable as well as $(\psi_2,\psi_1)$-locally transferrable, then they are simply called {\it locally transferrable}.

\item[(iii)] If there exist $\Lambda>0$ and $C>0$, such that the phase functions $\psi_1$ and $\psi_2$ satisfy
\begin{equation*}
\left|\psi_1(\lambda)-\psi_2(\lambda)\right| \le C\:, \text{ for all } \lambda > \Lambda\:,
\end{equation*} 
then the equations are said to be {\it of comparable oscillation}. 
\end{itemize}
\end{definition}

\begin{remark} \label{concavity_transference}
Point $(iii)$ of Definition \ref{my_defn} means that in high frequency, both the phase functions are within bounded error. Then by the condition (\ref{concavity}) in Definition \ref{asymp_concavity} and (\ref{example_phase}) in Remark \ref{example}, it follows that any asymptotically concave dispersive equation of degree $a \in (0,1)$ is {\it of comparable oscillation} to  the fractional Schr\"odinger equation of degree $a \in (0,1)$, corresponding to $\tilde{\Delta}$\:. 
\end{remark}

The notion of {\it comparable oscillation} defines an equivalence relation. The following result shows that each member in the same equivalence class satisfies the same local maximal estimates:  

\begin{lemma} \cite[Theorem $1.6$]{Dewan2} \label{transference_principle}
Let $\psi_1$ and $\psi_2$ be continuous real-valued functions on $[0,\infty)$ such that they are $C^\infty$ away from the origin. If the dispersive equations corresponding to $\psi_1$ and $\psi_2$ are of comparable oscillation, then they are also locally transferrable.  
\end{lemma} 
\subsection{A simplified version of the Euclidean result:} The solution of the fractional Schr\"odinger equation (\ref{frac_schrodinger}) for $a \in (0,1)$, on $\R^n$ with initial data $f \in \mathscr{S}(\R^n)_o$ is again radial and is given by,
\begin{equation} \label{solution} 
\tilde{S}_t f(r):= \int_{0}^\infty \J_{\frac{n-2}{2}}(\lambda r)\:e^{it\lambda^a}\:\mathscr{F}f(\lambda)\: \lambda^{n-1}\: d\lambda\:.
\end{equation}
The Homogeneous fractional $L^2$-Sobolev spaces on $\R^n$ are defined as, 
\begin{equation*}
\dot{H}^\beta(\R^n):=\left\{f \in L^2(S): {\|f\|}_{\dot{H}^\beta(\R^n)}:= {\left(\int_0^\infty \lambda^{2\beta}\: {|\mathscr{F}f(\lambda)|}^2 \lambda^{n-1} d\lambda\right)}^{1/2}< \infty\right\}\:.
\end{equation*}
Next for $q \ge 2$, set
\begin{equation*}
\delta(q)= \frac{c_2-1}{2}+n\left(\frac{1}{2}-\frac{1}{q}\right)\:,
\end{equation*}
where $c_2$ is as defined in (\ref{Walther_constants}). We now recall Walther's result on the unit ball in $\R^n$:
\begin{lemma}\cite[Theorem $2.7$]{Wa2} \label{Walther_result}
Let $a \in (0,1)\:,\:\beta \in \left(\frac{a}{4},\min\left\{\frac{a}{2},\frac{1}{4}\right\}\right)\:,\: q \in [2,2/c_2]$ and $n>1$\:. Then there is a positive number $C$ independent of $f \in \mathscr{S}(\R^n)_o$ such that the following inequality holds
\begin{equation}\label{Walther_estimate}
\left(\int_0^1 \left(\displaystyle\sup_{t \in (-1,1)}\left|\tilde{S}_t f(r)\right|\right)^q r^{q\delta(q)+n-1}\:dr\right)^{1/q} \le C\: {\|f\|}_{\dot{H}^\beta(\R^n)}\:.
\end{equation}
\end{lemma}

We now obtain a simplified version of Lemma \ref{Walther_result} for arbitrary balls centred at the origin, which will be useful in the proof of Theorem \ref{maximal_bddness_thm}: 
\begin{lemma}\label{Rn_result}
For $a \in (0,1)\:,\:\beta \in \left(\frac{a}{4},\min\left\{\frac{a}{2},\frac{1}{4}\right\}\right)\:,\:n>1$, any $R>0$ and any $f \in \mathscr{S}(\R^n)_o$, we have
\begin{equation}\label{Rn_estimate}
\left(\int_0^R \left(\displaystyle\sup_{t \in (0,1)}\left|\tilde{S}_t f(r)\right|\right)^2 r^{n-1}\:dr\right)^{1/2} \lesssim {\|f\|}_{\dot{H}^\beta(\R^n)}\:,
\end{equation}
where the implicit constant depends only on $R$ and $\beta$\:. 
\end{lemma}
\begin{proof}
By (\ref{Walther_constants_relations}) we have $c_2<1$ and hence $\delta(2)<0$\:. Then for $r \in (0,1)$, $r^{2\delta(2)}>1$ and hence (\ref{Walther_estimate}) implies that
\begin{equation}\label{eq1}
\left(\int_0^1 \left(\displaystyle\sup_{t \in (-1,1)}\left|\tilde{S}_t f(r)\right|\right)^2 r^{n-1}\:dr\right)^{1/2} \lesssim {\|f\|}_{\dot{H}^\beta(\R^n)}\:.
\end{equation}
For any $R \in (0,1)$, (\ref{Rn_estimate}) follows from (\ref{eq1}) as
\begin{eqnarray*}
\left(\int_0^R \left(\displaystyle\sup_{t \in (0,1)}\left|\tilde{S}_t f(r)\right|\right)^2 r^{n-1}\:dr\right)^{1/2} &\le & \left(\int_0^1 \left(\displaystyle\sup_{t \in (0,1)}\left|\tilde{S}_t f(r)\right|\right)^2 r^{n-1}\:dr\right)^{1/2} \\
& \le & \left(\int_0^1 \left(\displaystyle\sup_{t \in (-1,1)}\left|\tilde{S}_t f(r)\right|\right)^2 r^{n-1}\:dr\right)^{1/2}\:.
\end{eqnarray*}
Now for $R>1$, consider the dilated function $f_R(x):=f(Rx)$, for $x \in \R^n$. Now as the Fourier transforms of $f$ and $f_R$ are related as
\begin{equation} \label{eq2}
\mathscr{F}f(\lambda)= R^n \mathscr{F}f_R(R\lambda)\:,\: \lambda \in (0,\infty)\:,
\end{equation}
an elementary computation using the formula (\ref{solution}) yields 
\begin{equation} \label{eq3}
\tilde{S}_t f(r)= \tilde{S}_{t/R^a} f_R(r/R)\:,\:r \in [0,R]\:.
\end{equation}
Then for $R>1$, using (\ref{eq3}) and change of variables, we get
\begin{eqnarray} \label{eq4}
\left(\int_0^R \left(\displaystyle\sup_{t \in (0,1)}\left|\tilde{S}_t f(r)\right|\right)^2 r^{n-1}\:dr\right)^{1/2} &\le & \left(\int_0^R \left(\displaystyle\sup_{t \in (-R^a,R^a)}\left|\tilde{S}_t f(r)\right|\right)^2 r^{n-1}\:dr\right)^{1/2} \nonumber\\
&=& \left(\int_0^R \left(\displaystyle\sup_{t \in (-R^a,R^a)}\left|\tilde{S}_{\frac{t}{R^a}} f_R\left(\frac{r}{R}\right)\right|\right)^2 r^{n-1}\:dr\right)^{1/2} \nonumber\\
&=& R^{\frac{n}{2}}\left(\int_0^1 \left(\displaystyle\sup_{t \in (-1,1)}\left|\tilde{S}_t f_R(r)\right|\right)^2 r^{n-1}\:dr\right)^{1/2}\:.
\end{eqnarray}
Then applying (\ref{eq1}) to $f_R$ and combining it with (\ref{eq4}), it follows that
\begin{equation}\label{eq5}
\left(\int_0^R \left(\displaystyle\sup_{t \in (0,1)}\left|\tilde{S}_t f(r)\right|\right)^2 r^{n-1}\:dr\right)^{1/2} \lesssim R^{\frac{n}{2}}{\|f_R\|}_{\dot{H}^\beta(\R^n)}\:.
\end{equation}
Finally since by (\ref{eq2}), the Homogeneous Sobolev norms of $f$ and $f_R$ are related as
\begin{equation*}
{\|f_R\|}_{\dot{H}^\beta(\R^n)} = R^{\beta-\frac{n}{2}}{\|f\|}_{\dot{H}^\beta(\R^n)}\:,
\end{equation*}
plugging this in (\ref{eq5}) yields 
\begin{equation*}
\left(\int_0^R \left(\displaystyle\sup_{t \in (0,1)}\left|\tilde{S}_t f(r)\right|\right)^2 r^{n-1}\:dr\right)^{1/2} \lesssim R^\beta{\|f\|}_{\dot{H}^\beta(\R^n)}\:.
\end{equation*}
This completes the proof of Lemma \ref{Rn_result}.
\end{proof}
\section{Proof of Theorem \ref{maximal_bddness_thm}}
By the transference principle (Lemma \ref{transference_principle}) and Remark \ref{concavity_transference}, Theorem \ref{maximal_bddness_thm} for a general asymptotically concave dispersive equation of degree $a \in (0,1)$, follows from the result for the special case of the fractional Schr\"odinger equation with degree $a \in (0,1)$, corresponding to $\tilde{\Delta}$. Hence it suffices to give the proof of Theorem \ref{maximal_bddness_thm} for this special case.

We fix $a \in (0,1)$. To lighten the notation, we denote the maximal function for the fractional Schr\"odinger equation with degree $a$, corresponding to $\tilde{\Delta}$, simply by $S^*$. For $f \in \mathscr{S}^2(S)_o$\:, we will prove the stronger estimate in terms of the homogeneous Sobolev norm:
\begin{equation} \label{stronger_estimate}
{\|S^* f\|}_{L^2\left(B_R\right)} \lesssim {\|f\|}_{\dot{H}^{\beta}(S)}\:,\text{ for }\: \beta \in I_a \:,
\end{equation}
where 
\begin{equation*}
I_a:=\left(\frac{a}{4},\min\left\{\frac{a}{2},\frac{1}{4}\right\}\right)\:.
\end{equation*}
Then by (\ref{sobolev_inclusion}), Theorem \ref{maximal_bddness_thm} follows. In order to prove the estimate (\ref{stronger_estimate}), it suffices to prove the following estimate in terms of the linearized maximal function,
\begin{equation} \label{linearized_maximal_estimate}
{\|Tf\|}_{L^2\left(B_R\right)} \lesssim {\|f\|}_{\dot{H}^{\beta}(S)}\:,\text{ for }\: \beta \in I_a \:,
\end{equation}
where,
\begin{equation*}
T f(s) := \int_{0}^\infty \varphi_\lambda(s)\:e^{it(s)\lambda^a}\:\widehat{f}(\lambda)\: {|{\bf c}(\lambda)|}^{-2}\: d\lambda \:,
\end{equation*}
and $t(\cdot): [0,\infty) \to (0,1)$ is a measurable function. The key strategy in the proof of (\ref{linearized_maximal_estimate}) is to invoke the various series expansions of $\varphi_\lambda$. Now as the validity of these series expansions depend on the geodesic distance from the group identity, it prompts us to decompose the ball $B_R$ into a closed ball and an open annulus, $$B_R = \overline{B_{R_0}} \sqcup A(R_0,R)\:,$$
where
\begin{equation*}
\overline{B_{R_0}}=\{x \in S : d(e,x)\leq R_0 \}\:\text{ and } A(R_0,R) = \{x \in S : R_0 < d(e,x) < R\}\:,
\end{equation*}
and then to estimate $T$ on $\overline{B_{R_0}}$ and $A(R_0,R)$ separately.

\subsection{Estimating $T$ on $\overline{B_{R_0}}$:} We first decompose $T$ in terms of small and large frequency. More precisely, we consider a non-negative, even function $\eta \in C^\infty_c(\R)$ such that $Supp(\eta) \subset \left\{\xi : 1/2 < |\xi| <2\right\}$ and 
\begin{equation*}
\displaystyle\sum_{k=-\infty}^\infty \eta \left(2^{-k} \xi\right)=1\:,\: \text{ for } \xi \ne 0\:.
\end{equation*}
Let us now define,
\begin{equation*}
\eta_1(\xi):= \displaystyle\sum_{k=-\infty}^0 \eta(2^{-k} \xi) \:,\text{   and  } \eta_2(\xi):= \displaystyle\sum_{k=1}^\infty \eta(2^{-k} \xi) \:.
\end{equation*}
We note that both $\eta_1$ and $\eta_2$ are even non-negative smooth functions with $Supp(\eta_1) \subset (-2,2)$, $Supp(\eta_2) \subset \R \setminus (-1,1)$ and $\eta_1+\eta_2 \equiv 1$. Accordingly, we decompose $T$ as,
\begin{equation}\label{first_decomposition}
Tf(s)= T_1 f(s) + T_2 f(s)\:,
\end{equation}
where,
\begin{eqnarray*}
T_1 f(s):= \int_{0}^2 \varphi_\lambda(s)\:e^{it(s)\lambda^a}\:\widehat{f}(\lambda)\:\eta_1(\lambda)\: {|{\bf c}(\lambda)|}^{-2}\: d\lambda\:, \\
T_2 f(s):= \int_{1}^\infty \varphi_\lambda(s)\:e^{it(s)\lambda^a}\:\widehat{f}(\lambda)\:\eta_2(\lambda)\: {|{\bf c}(\lambda)|}^{-2}\: d\lambda\:.
\end{eqnarray*}
Note that $T_1$ and $T_2$ correspond to the small and large frequency respectively.

We first estimate $T_1$. For $s \in [0,R_0]$, by the boundedness of $\varphi_\lambda$ and the multiplier, followed by an application of the Cauchy-Schwarz inequality, we get
\begin{eqnarray*}
\left|T_1f(s)\right| & \le & \int_{0}^2 \left|\widehat{f}(\lambda)\right|\: {|{\bf c}(\lambda)|}^{-2}\: d\lambda \\
&\le & {\|f\|}_{\dot{H}^{\beta}(S)} \left(\int_{0}^2 \frac{{|{\bf c}(\lambda)|}^{-2}}{\lambda^{2\beta}}d\lambda\right)^{1/2}\:.
\end{eqnarray*}
Now as $\beta \in I_a$, in particular, $2\beta<1$. Then combining this with the small frequency asymptotics of ${|{\bf c}(\cdot)|}^{-2}$ given by (\ref{plancherel_measure}), it follows that
\begin{equation*}
\left|T_1f(s)\right| \lesssim  {\|f\|}_{\dot{H}^{\beta}(S)}\:,\:\: s \in [0,R_0]\:,
\end{equation*}
which in turn implies that
\begin{equation} \label{first_estimate}
{\|T_1f\|}_{L^2\left(\overline{B_{R_0}}\right)} \lesssim {\|f\|}_{\dot{H}^{\beta}(S)}\:,\text{ for }\: \beta \in I_a \:.
\end{equation}

Now to estimate $T_2$, we make use of the Bessel series expansion of $\varphi_\lambda$ (Lemma \ref{bessel_series_expansion}). For $s\in [0,R_0]$, putting $M=0$ in Lemma \ref{bessel_series_expansion}, we get
\begin{equation} \label{bessel_expansion}
\varphi_\lambda(s)= c_0 {\left(\frac{s^{n-1}}{A(s)}\right)}^{1/2} \J_{\frac{n-2}{2}}(\lambda s) + E_1(\lambda,s),\:\:\:\:\:\:\:\:\:\:\lambda \ge 1\:,
\end{equation}
where
\begin{equation} \label{bessel_error_estimate}
	\left|E_1(\lambda,s) \right| \lesssim  \begin{cases}
	 s^2,  & \text{ if  }\: \lambda s \le 1 \\
	s^2 {(\lambda s)}^{-\left(\frac{n+1}{2} \right)}, &\text{ if  }\: \lambda s > 1 \:.
	\end{cases}
\end{equation}
Then invoking the decomposition of $\varphi_\lambda$ given in (\ref{bessel_expansion}), we further decompose $T_2f$ as follows
\begin{equation} \label{second_decomposition}
T_2 f(s)=T_3 f(s)+T_4 f(s)\:,
\end{equation}
where
\begin{eqnarray*}
T_3 f(s) &=& c_0 {\left(\frac{s^{n-1}}{A(s)}\right)}^{1/2} \int_{1}^\infty \J_{\frac{n-2}{2}}(\lambda s)\:e^{it(s)\lambda^a}\:\widehat{f}(\lambda)\:\eta_2(\lambda)\: {|{\bf c}(\lambda)|}^{-2}\: d\lambda\:, \\
T_4 f(s)&=& \int_{1}^\infty E_1(\lambda, s)\:e^{it(s)\lambda^a}\:\widehat{f}(\lambda)\:\eta_2(\lambda)\: {|{\bf c}(\lambda)|}^{-2}\: d\lambda\:. 
\end{eqnarray*}

To estimate $T_3$, we note that as $Supp(\widehat{f}\cdot \eta_2) \subset [1,\infty)$, by the Schwartz isomorphism theorem  there exists a unique $F \in \mathscr{S}^2(S)^\infty_o$ (defined as in (\ref{new_schwartz_spaces})) such that
\begin{equation} \label{piece3_eq1}
\widehat{f}(\lambda) \eta_2(\lambda) = \widehat{F}(\lambda)\:.
\end{equation}
Then by Lemma \ref{schwartz_correspondence} there exists a unique $g \in \mathscr{S}(\R^n)^\infty_o$ such that
\begin{equation} \label{euclidean_support}
Supp(\mathscr{F}g)=Supp(\widehat{F}) \subset [1,\infty)\:,
\end{equation}
with $\mathscr{F}g$ and $\widehat{F}$ related as,
\begin{equation} \label{piece3_eq2}
\lambda^{n-1}\:\mathscr{F}g(\lambda)=\:{|{\bf c}(\lambda)|}^{-2}\:\widehat{F}(\lambda) \:.
\end{equation}
Plugging the relation (\ref{piece3_eq2}) in the definition of the Homogeneous Sobolev norms and then using the large frequency asymptotics of ${|{\bf c}(\cdot)|}^{-2}$ given in (\ref{plancherel_measure}), it follows that the Homogeneous Sobolev norms of $g$ and $F$ are comparable, that is for all $\beta \ge 0$,
\begin{equation} \label{piece3_eq3}
\|g\|_{\dot{H}^\beta(\R^n)} \asymp \|F\|_{\dot{H}^\beta(S)}\:.
\end{equation} 
We now consider $\tilde{T}$, the linearized maximal function  corresponding to the fractional Schr\"odinger equation (\ref{frac_schrodinger}) for $a \in (0,1)$, on $\R^n$. $\tilde{T}$ is defined on $\mathscr{S}(\R^n)_o$, for $s \in [0,\infty)$  by,
\begin{equation*}
\tilde{T}h(s)= \int_{0}^\infty \J_{\frac{n-2}{2}}(\lambda s)\:e^{it(s)\lambda^a}\:\mathscr{F}h(\lambda)\: \lambda^{n-1}\: d\lambda,\:\text{ for } h \in  \mathscr{S}(\R^n)_o\:,
\end{equation*}
where $t(\cdot)$ is the same measurable function on $[0,\infty)$. Now using the relations (\ref{piece3_eq1})-(\ref{piece3_eq2}) and the local growth asymptotics of the density function (\ref{density_function}) for $s \in [0,R_0]$, we obtain a pointwise comparison of the moduli of $T_3 f$ and $\tilde{T}g$:
\begin{eqnarray} \label{piece3_eq4}
|T_3 f(s)| & \asymp & \left|\int_{1}^\infty \J_{\frac{n-2}{2}}(\lambda s)\:e^{it(s)\lambda^a}\:\widehat{f}(\lambda)\:\eta_2(\lambda)\: {|{\bf c}(\lambda)|}^{-2}\: d\lambda\right|\nonumber\\
&=& \left|\int_{1}^\infty \J_{\frac{n-2}{2}}(\lambda s)\:e^{it(s)\lambda^a}\:\widehat{F}(\lambda)\: {|{\bf c}(\lambda)|}^{-2}\: d\lambda\right|\nonumber\\
&=& \left|\int_{1}^\infty \J_{\frac{n-2}{2}}(\lambda s)\:e^{it(s)\lambda^a}\:\mathscr{F}g(\lambda)\: \lambda^{n-1}\: d\lambda\right|\nonumber \\
& = & |\tilde{T}g(s)|\:.
\end{eqnarray}
Then by (\ref{piece3_eq4}), the local growth asymptotics of the density function (\ref{density_function}), Lemma \ref{Rn_result} and (\ref{piece3_eq3}), it follows for $\beta \in I_a$,
\begin{equation} \label{second_estimate}
\|T_3 f\|_{L^2(\overline{B_{R_0}})} \lesssim \|\tilde{T} g\|_{L^2(\overline{B(o,{R_0})})} \lesssim \|g\|_{\dot{H}^\beta(\R^n)} \lesssim \|F\|_{\dot{H}^\beta(S)} \le \|f\|_{\dot{H}^\beta(S)} \:.
\end{equation}

To estimate $T_4$, we make use of the pointwise bounds of $E_1$ given by (\ref{bessel_error_estimate}). To see this, we first consider when $s \in [1,R_0]$. In this case, by (\ref{bessel_error_estimate}), an application of the Cauchy-Schwarz inequality and the large frequency asymptotics of ${|{\bf c}(\cdot)|}^{-2}$ given in (\ref{plancherel_measure}), it follows for $\beta \in I_a$,
\begin{eqnarray} \label{piece4_eq1}
|T_4f(s)| &\lesssim & s^{-\frac{n-3}{2}} \int_1^\infty \lambda^{-\frac{n+1}{2}}\:|\what{f}(\lambda)|\:{|{\bf c}(\lambda)|}^{-2}\: d\lambda \nonumber\\
&\lesssim & \|f\|_{\dot{H}^\beta(S)} \left(\int_1^\infty \frac{d \lambda}{\lambda^{2(\beta +1)}}\right)^{1/2} \nonumber\\
&\lesssim & \|f\|_{\dot{H}^\beta(S)}\:.
\end{eqnarray}
For $s \in [0,1]$, again by (\ref{bessel_error_estimate}) and an application of the Cauchy-Schwarz inequality,
\begin{eqnarray} \label{piece4_eq2}
|T_4f(s)| &\lesssim & s^2 \int_1^{\frac{1}{s}} |\what{f}(\lambda)|\:{|{\bf c}(\lambda)|}^{-2}\: d\lambda +  s^{-\frac{n-3}{2}} \int_{\frac{1}{s}}^\infty \lambda^{-\frac{n+1}{2}}\:|\what{f}(\lambda)|\:{|{\bf c}(\lambda)|}^{-2}\: d\lambda \nonumber\\
&\lesssim & \|f\|_{\dot{H}^\beta(S)} \left[s^2\: \mathcal{I}_1(s)^{1/2}\:+\: s^{-\frac{n-3}{2}}\:\mathcal{I}_2(s)^{1/2}\right] \:,
\end{eqnarray}
where 
\begin{equation*}
\mathcal{I}_1(s) = \int_1^{\frac{1}{s}} \frac{{|{\bf c}(\lambda)|}^{-2}}{\lambda^{2\beta}}\:d\lambda\:\:\text{ and }\:\: \mathcal{I}_2(s)= \int_{\frac{1}{s}}^\infty \frac{{|{\bf c}(\lambda)|}^{-2}}{\lambda^{2\beta+n+1}}\:d\lambda\:.
\end{equation*}
Then as for $\beta \in I_a\:,\: n-2\beta-1>0$, it follows from the large frequency asymptotics of ${|{\bf c}(\cdot)|}^{-2}$ given in (\ref{plancherel_measure}) that
\begin{equation} \label{piece4_eq3}
\mathcal{I}_1(s) \lesssim \int_1^{\frac{1}{s}} \lambda^{n-2\beta-1}\:d\lambda \lesssim s^{-(n-2\beta)}\:,
\end{equation} 
\begin{equation} \label{piece4_eq4}
\mathcal{I}_2(s) \lesssim \int_{\frac{1}{s}}^\infty \frac{d\lambda}{\lambda^{2(\beta+1)}} \lesssim s^{2\beta +1}\:.
\end{equation}
Then plugging (\ref{piece4_eq3}) and (\ref{piece4_eq4}) in (\ref{piece4_eq2}), we get that for $s \in [0,1]$ and $\beta \in I_a$,
\begin{equation} \label{piece4_eq5}
|T_4f(s)| \lesssim s^{\beta +2 -\frac{n}{2}}\: \|f\|_{\dot{H}^\beta(S)}\:.
\end{equation}
Then combining (\ref{piece4_eq1}) and (\ref{piece4_eq5}), for $\beta \in I_a$, it follows from the local growth asymptotics of the density function (\ref{density_function}) that
\begin{equation} \label{third_estimate}
\|T_4 f\|_{L^2(\overline{B_{R_0}})} \lesssim \|f\|_{\dot{H}^\beta(S)} \:\left(\int_0^1 s^{2\beta +3}\:ds \:+\: \int_1^{R_0} s^{n-1}\:ds\right)^{1/2} \lesssim \|f\|_{\dot{H}^\beta(S)}\:.
\end{equation}
Now the decomposition (\ref{second_decomposition}) along with the estimates (\ref{second_estimate}) and (\ref{third_estimate}) yield for $\beta \in I_a$,
\begin{equation*} 
\|T_2 f\|_{L^2(\overline{B_{R_0}})} \lesssim \|f\|_{\dot{H}^\beta(S)}\:,
\end{equation*}
which along with the decomposition (\ref{first_decomposition}) and the estimate (\ref{first_estimate}) yield the desired estimate of $T$:
\begin{equation} \label{estimate_small_ball} 
\|T f\|_{L^2(\overline{B_{R_0}})} \lesssim \|f\|_{\dot{H}^\beta(S)}\:,\:\:\beta \in I_a\:.
\end{equation}
\subsection{Estimating $T$ on $A(R_0,R)$:} For $R_0 < s < R$ and $\lambda >0$, using the series expansion  (\ref{anker_series_expansion}) and the estimate (\ref{coefficient_estimate}) on the coefficients $\Gamma_\mu$, we get
\begin{equation} \label{annulus_pf_eq1}
\varphi_\lambda (s) = 2^{-\frac{m_\z}{2}} {A(s)}^{-\frac{1}{2}} \left\{{\bf c}(\lambda)  e^{i \lambda s} + {\bf c}(-\lambda) e^{-i\lambda  s}\right\} + \mathscr{E}_2(\lambda,s)\:,
\end{equation}
where
\begin{equation*}
\mathscr{E}_2(\lambda,s) = 2^{-\frac{m_\z}{2}} {A(s)}^{-\frac{1}{2}} \left\{{\bf c}(\lambda)  \displaystyle\sum_{\mu=1}^\infty \Gamma_\mu(\lambda) e^{(i \lambda-\mu) s} + {\bf c}(-\lambda) \displaystyle\sum_{\mu=1}^\infty \Gamma_\mu(-\lambda) e^{-(i\lambda + \mu) s}\right\} \:,
\end{equation*}
and thus
\begin{equation} \label{annulus_pf_eq2}
\left|\mathscr{E}_2(\lambda,s)\right| \lesssim {A(s)}^{-\frac{1}{2}} \left|{\bf c}(\lambda)\right| {(1+\lambda)}^{-1} \:.
\end{equation}
Then invoking (\ref{annulus_pf_eq1}) in the formula defining the linearized maximal function and using the fact that ${|{\bf c}(\cdot)|}^{-2}$ is even on $\R \setminus \{0\}$, we get
\begin{eqnarray} \label{third_decomposition}
Tf(s) &=& 2^{-\frac{m_\z}{2}} {A(s)}^{-\frac{1}{2}} \int_{\R \setminus \{0\}} {\bf c}(\lambda)\:e^{i(\lambda s+t(s)|\lambda|^a)}\:\what{f}(\lambda)\:{|{\bf c}(\lambda)|}^{-2}\:d\lambda \nonumber\\
&& + \int_0^\infty \mathscr{E}_2(\lambda, s)\:e^{it(s)\lambda^a}\:\widehat{f}(\lambda)\: {|{\bf c}(\lambda)|}^{-2}\: d\lambda \nonumber\\
&=& T_5f(s)\:+\:T_6f(s)\:.
\end{eqnarray}
To obtain the estimate 
\begin{equation} \label{annulus_first_estimate} 
\|T_5 f\|_{L^2(A(R_0,R))} \lesssim \|f\|_{\dot{H}^\beta(S)}\:,\:\:\beta \in I_a\:,
\end{equation}
we first note that
\begin{eqnarray*}
T_5 f(s)\:{A(s)}^{\frac{1}{2}}&=&2^{-\frac{m_\z}{2}}  \int_{\R \setminus \{0\}} {\bf c}(\lambda)\:e^{i(\lambda s+t(s)|\lambda|^a)}\:\what{f}(\lambda)\:{|{\bf c}(\lambda)|}^{-2}\:d\lambda \\
&=& 2^{-\frac{m_\z}{2}}  \int_{\R \setminus \{0\}} {\bf c}(\lambda)\:e^{i(\lambda s+t(s)|\lambda|^a)}\:g(\lambda)\:|\lambda|^{-\beta}\:{|{\bf c}(\lambda)|}^{-1}\:d\lambda\:,
\end{eqnarray*}
where 
\begin{equation*}
g(\lambda) = \what{f}(\lambda)\:|\lambda|^\beta\:{|{\bf c}(\lambda)|}^{-1}\:,\:\: \lambda \in \R \setminus \{0\}\:.
\end{equation*}
Then setting
\begin{equation*}
Pg(s):= \int_{\R \setminus \{0\}} {\bf c}(\lambda)\:e^{i(\lambda s+t(s)|\lambda|^a)}\:g(\lambda)\:|\lambda|^{-\beta}\:{|{\bf c}(\lambda)|}^{-1}\:d\lambda\:,
\end{equation*}
that is,
\begin{equation*}
2^{\frac{m_\z}{2}}\:T_5 f(s)\:{A(s)}^{\frac{1}{2}}=Pg(s)\:,
\end{equation*}
we note that to prove (\ref{annulus_first_estimate}), it suffices to prove that 
\begin{equation} \label{annulus_es1_eq1}
\left(\int_{R_0}^R |Pg(s)|^2\:ds\right)^{\frac{1}{2}} \lesssim \left(\int_{\R \setminus \{0\}} |g(\lambda)|^2\:d\lambda\right)^{\frac{1}{2}}\:. 
\end{equation}
Now consider an even $\rho \in C^\infty_c(\R)$ which takes values in $[0,1]$, such that
\begin{eqnarray*}
\rho(\lambda)&=& 1,\:\:\:\:|\lambda| < 1,\\
&=&0,\:\:\:\:|\lambda| \ge 2\:.
\end{eqnarray*}
Then for $N>2$, $s \in (R_0,R)$, we set
\begin{equation*}
P_N g(s):= \int_{\R \setminus \{0\}} {\bf c}(\lambda)\:e^{i(\lambda s+t(s)|\lambda|^a)}\:\rho\left(\frac{\lambda}{N}\right)\:g(\lambda)\:|\lambda|^{-\beta}\:{|{\bf c}(\lambda)|}^{-1}\:d\lambda \:.
\end{equation*}
Now for $u \in C^\infty_c(R_0,R)$ and $\lambda \in \R \setminus \{0\}$, setting
\begin{equation*}
P^*_N u(\lambda)= \overline{{\bf c}(\lambda)}\:\rho\left(\frac{\lambda}{N}\right)\:|\lambda|^{-\beta}\:{|{\bf c}(\lambda)|}^{-1} \int_{R_0}^R e^{-i(\lambda s+t(s)|\lambda|^a)}\:u(s)\:ds\:,
\end{equation*}
it is easy to see that
\begin{equation*}
\int_{R_0}^R P_N v(s)\: \overline{u(s)}\: ds = \int_{\R \setminus \{0\}} v(\lambda)\: \overline{P^*_N u(\lambda)}\: d\lambda\:,
\end{equation*}
holds for all $u \in C^\infty_c(R_0,R)$ and $v \in L^2(\R \setminus \{0\})$ having suitable decay at infinity. Thus it suffices to prove that
\begin{equation} \label{annulus_es1_eq2}
{\left(\int_{\R \setminus \{0\}} {\left|P^*_N h(\lambda)\right|}^2 d\lambda\right)}^{\frac{1}{2}} \lesssim {\left(\int_{R_0}^R {|h(s)|}^2\: ds\right)}^{\frac{1}{2}}\:,\text{ for all }h \in C^\infty_c(R_0,R),
\end{equation}
with the implicit constant independent of $N$, as then letting $N \to \infty$, by duality we can obtain the estimate (\ref{annulus_es1_eq1}). Now by Fubini's theorem,
\begin{equation*}
\int_{\R \setminus \{0\}} {\left|P^*_N h(\lambda)\right|}^2 d\lambda = \int_{R_0}^R \int_{R_0}^R \mathcal{I}_N(s,s')\: h(s)\: \overline{h(s')}\: ds\: ds',
\end{equation*}
where
\begin{equation*}
\mathcal{I}_N(s,s') =  \int_{\R \setminus \{0\}} e^{i\left\{(s'-s)\lambda+(t(s')-t(s))|\lambda|^a\right\}}\: |\lambda|^{-2\beta}\: \rho\left(\frac{\lambda}{N}\right)^2\: d\lambda \:.
\end{equation*}
Then noting that $|s'-s| \le (R-R_0)$, $|t(s')-t(s)|\le 1$, we apply the oscillatory integral estimate given in Lemma \ref{oscillatory_integral_estimate} and use the relation (\ref{Walther_constants_relations}) to obtain
\begin{equation*}
\int_{\R \setminus \{0\}} {\left|P^*_N h(\lambda)\right|}^2 d\lambda \lesssim \int_{R_0}^R \int_{R_0}^R \frac{1}{{|s-s'|}^{c_2}}\: |h(s)|\: |h(s')|\: ds\: ds',
\end{equation*}
with the implicit constant independent of $N$. 
As $h \in C^\infty_c(R_0,R)$, we can think of it as an even $C^\infty_c$ function supported in $(-R,-R_0) \sqcup (R_0,R)$. We can therefore write the last integral as a one dimensional Riesz potential and apply (\ref{riesz_identity}) to get that for some positive number $c$,
\begin{eqnarray*}
\int_{R_0}^R \int_{R_0}^R \frac{1}{{|s-s'|}^{c_2}}\: |h(s)|\: |h(s')|\: ds\: ds' &=&  c\int_{0}^\infty I_{1-c_2} (|h|)(s)\:|h(s)|\:ds \\
&=& c \int_{\R} {|\xi|}^{c_2-1}\: {\left|\tilde{h}(\xi)\right|}^2\: d\xi \:.
\end{eqnarray*}
Now setting $\alpha:=(1-c_2)/2, p=2$, we note that $\alpha_1:= \alpha + \frac{1}{2} - \frac{1}{p} =\alpha$. Then as by (\ref{Walther_constants_relations}) $c_2 \in (0,1)$, it follows that $0<\alpha=\alpha_1<1/2$ and thus both the conditions in (\ref{Pitt's_ineq_cond}) are satisfied. Then by Pitt's inequality (\ref{Pitt's_ineq}), we get that
\begin{eqnarray*}
 \int_{\R} {|\xi|}^{c_2-1}\: {\left|\tilde{h}(\xi)\right|}^2\: d\xi
 &\lesssim &  \int_{\R} {|h(x)|}^2\: {|x|}^{1-c_2}\: dx \\
&=& c \int_{R_0}^R {|h(s)|}^{2}\: s^{1-c_2}\: ds 
 \le   c \:R^{1-c_2} \: {\|h\|}^2_{{L^2}(R_0,R)}\:.
\end{eqnarray*}
Thus we get (\ref{annulus_es1_eq2}) and consequently, the desired estimate (\ref{annulus_first_estimate}) for $T_5$. 

Finally for $T_6$, using the estimate on the remainder term (\ref{annulus_pf_eq2}) and Cauchy-Schwarz inequality, we obtain
\begin{eqnarray*}
\left|T_6 f(s)\right| & \lesssim & {A(s)}^{-\frac{1}{2}}\: \int_0^\infty  {(1+\lambda)}^{-1}\: \left|\widehat{f}(\lambda)\right|\: {|{\bf c}(\lambda)|}^{-1}\: d\lambda \\
& \le & {A(s)}^{-\frac{1}{2}}\: \left[\int_0^1 \left|\widehat{f}(\lambda)\right|\: {|{\bf c}(\lambda)|}^{-1}\: d\lambda + \int_1^\infty \lambda^{-1}\:\left|\widehat{f}(\lambda)\right|\: {|{\bf c}(\lambda)|}^{-1}\: d\lambda  \right] \\
& \le & {A(s)}^{-\frac{1}{2}}\: {\|f\|}_{\dot{H}^{\beta}(S)} \: \left[{\left(\int_0^1 \frac{d\lambda}{\lambda^{2\beta}}\right)}^{1/2} + {\left(\int_1^\infty \frac{d\lambda}{\lambda^{2(1+\beta)}}\right)}^{1/2}\right] \:.
\end{eqnarray*}
Now as $\beta \in I_a$, we have $0<2\beta<1$ and thus both the integrals on the right hand side are integrable. Hence, we get
\begin{equation*}
\left|T_6 f(s)\right| \lesssim {A(s)}^{-\frac{1}{2}}\: {\|f\|}_{\dot{H}^{\beta}(S)}\:,
\end{equation*}
which in turn yields
\begin{equation} \label{annulus_second_estimate}
\|T_6 f\|_{L^2(A(R_0,R))} \lesssim {\|f\|}_{\dot{H}^{\beta}(S)}  \:,\:\: \beta \in I_a\:.
\end{equation}
Now the decomposition (\ref{third_decomposition}) along with the estimates (\ref{annulus_first_estimate}) and (\ref{annulus_second_estimate}) yield 
\begin{equation*} 
\|T f\|_{L^2(A(R_0,R))} \lesssim {\|f\|}_{\dot{H}^{\beta}(S)}  \:,\:\: \beta \in I_a\:,
\end{equation*}
which along with (\ref{estimate_small_ball}) yields (\ref{linearized_maximal_estimate}) and completes the proof of Theorem \ref{maximal_bddness_thm}.

\section{Proof of Theorem \ref{sharpness_thm}}
Fix $a \in (0,1)$. We first work out the case of the Fractional Schr\"odinger equation of degree $a$, corresponding to $\tilde{\Delta}$. For $f \in \mathscr{S}^2(S)_o$, we recall the linearized maximal function,
\begin{equation*}
Tf(s) = \int_{0}^\infty \varphi_\lambda(s)\:e^{it(s)\lambda^a}\:\widehat{f}(\lambda)\: {|{\bf c}(\lambda)|}^{-2}\: d\lambda\:.
\end{equation*}
It suffices to prove that there is no inequality of the form,
\begin{equation} \label{counter_eq1}
\|T f\|_{L^2(B_1)} \lesssim {\|f\|}_{H^\beta(S)}\:,
\end{equation}
for any $\beta <a/4$. We consider an even function $\eta \in C^\infty_c(\R)$ which takes values in $[0,1]$, such that
\begin{eqnarray*}
\eta(\lambda)&=& 1,\:\:\:\:|\lambda| \le \frac{1}{2}\:,\\
&=&0,\:\:\:\:|\lambda| \ge 1\:.
\end{eqnarray*}
Then for integers $N \ge 2$, choose $f_N \in \mathscr{S}^2(S)_o$ such that 
\begin{equation*}
\widehat{f_N}(\lambda)=N^{-\frac{1}{2}}\: \eta \left(-N^{\left(\frac{a}{2}-1\right)} \lambda + N^{\frac{a}{2}}\right)|{\bf c}(\lambda)|\:.
\end{equation*}
It follows that $Supp(\widehat{f_N}) \subset \left(N-N^{\left(1-\frac{a}{2}\right)}\:,\:N+N^{\left(1-\frac{a}{2}\right)}\right)$ and 
\begin{eqnarray*}
{\|f_N\|}_{H^\beta(S)} &=& {\left(\int_0^\infty {\left(\lambda^2 + \frac{Q^2}{4}\right)}^\beta \: \widehat{f_N}(\lambda)^2 \: {|{\bf c}(\lambda)|}^{-2} d\lambda\right)}^{\frac{1}{2}} \\
&=& {\left(\frac{1}{N}\int_{N-N^{\left(1-\frac{a}{2}\right)}}^{N+N^{\left(1-\frac{a}{2}\right)}} {\left(\lambda^2 + \frac{Q^2}{4}\right)}^\beta \: \eta \left(-N^{\left(\frac{a}{2}-1\right)} \lambda + N^{\frac{a}{2}}\right)^2 \: d\lambda\right)}^{\frac{1}{2}} \\
& \lesssim & {\left(\frac{1}{N}\int_{N-N^{\left(1-\frac{a}{2}\right)}}^{N+N^{\left(1-\frac{a}{2}\right)}} N^{2\beta} \: d\lambda\right)}^{1/2} \\
& \lesssim & N^{\beta - \frac{a}{4}}\:.
\end{eqnarray*}
Thus, we have
\begin{equation} \label{counter_eq2}
{\|f_N\|}_{H^\beta(S)} \to 0\:,\:\text{ as } N \to \infty\:,\:\text{ if } \beta <a/4\:.
\end{equation}
Hence in view of (\ref{counter_eq2}), in order to prove the failure of (\ref{counter_eq1}), it suffices to show that there exists $N_0 \in \N$ such that for all $N > N_0$, 
\begin{equation} \label{counter_eq3}
\|T f_N\|_{L^2(B_1)} > c_3\:,
\end{equation}
for some constant $c_3>0$, independent of $N$\:. To prove $(\ref{counter_eq3})$, we consider for $\varepsilon \in (0,1/2)$ and integers $N \ge 2$, the sets,
\begin{equation*}
A_{N,\:\varepsilon}:= \left\{s \mid a\varepsilon N^{a-1}< s <2a\varepsilon N^{a-1}\right\}\:.
\end{equation*}
Now as $a \in (0,1)$, for $\varepsilon \in (0,1/2)$ and $N \ge 2$, we note that $A_{N,\:\varepsilon} \subset (0,1)$. Henceforth in our proof, we work under the assumption that $s \in A_{N,\:\varepsilon}$. By putting $M=0$ in Lemma \ref{bessel_series_expansion}, we expand $\varphi_\lambda$ in the Bessel series: 
\begin{equation} \label{counter_eq4}
\varphi_\lambda(s)= c_0 {\left(\frac{s^{n-1}}{A(s)}\right)}^{1/2} \J_{\frac{n-2}{2}}(\lambda s) + E_1(\lambda,s)\:.
\end{equation}
We recall the definition,
\begin{equation} \label{counter_eq5}
\J_{\frac{n-2}{2}}(\lambda s)= 2^{\frac{n-2}{2}} \: \pi^{1/2} \: \Gamma\left(\frac{n-1}{2}\right) \frac{J_{\frac{n-2}{2}}(\lambda s)}{(\lambda s)^{\left(\frac{n-2}{2}\right)}}\:.
\end{equation}
Now choosing an integer $N_1 \ge 2$, so that
\begin{equation*}
N^a_1 - N^{a/2}_1 > \frac{B_n}{a\varepsilon}\:,
\end{equation*}
where $B_n=\max\left\{A_{\frac{n-2}{2}},1\right\}$ ($A_{\frac{n-2}{2}}$ is as in Lemma \ref{bessel_function_expansion}), note that for $N \ge N_1$,  we have for $\lambda \in Supp(\widehat{f_N})$,
\begin{equation*}
\frac{B_n}{s} < \frac{B_n}{a\varepsilon N^{a-1}} < N- N^{\left(1-\frac{a}{2}\right)} < \lambda \:.
\end{equation*} 
Thus for $s \in A_{N,\:\varepsilon}$ and $\lambda \in Supp(\widehat{f_N})$, as $\lambda s>1$, by the estimates (\ref{bessel_error_estimate}) we get that
\begin{equation*}
\left|E_1(\lambda,s)\right| \lesssim s^2\: {(\lambda s)}^{-\frac{n+1}{2}}\:.
\end{equation*}
Moreover, as $\lambda s> A_{\frac{n-2}{2}}$, by Lemma \ref{bessel_function_expansion},
\begin{eqnarray} \label{counter_eq6}
J_{\frac{n-2}{2}}(\lambda s) &=& \sqrt{\frac{2}{\pi}} \frac{1}{{(\lambda s)}^{1/2}} \cos \left(\lambda s - \frac{\pi}{2}\left(\frac{n-2}{2}\right) - \frac{\pi}{4}\right) + \tilde{E}_{\frac{n-2}{2}}(\lambda s) \nonumber \\
&=& \sqrt{\frac{1}{2\pi}} \frac{1}{{(\lambda s)}^{1/2}} \left\{e^{i\left(\lambda s - \frac{\pi}{4}(n-1)\right)} + e^{-i\left(\lambda s - \frac{\pi}{4}(n-1)\right)}\right\} + \tilde{E}_{\frac{n-2}{2}}(\lambda s)\:,
\end{eqnarray}
where
\begin{equation*}
\left|\tilde{E}_{\frac{n-2}{2}}(\lambda s)\right| \lesssim \frac{1}{{(\lambda s)}^{3/2}}\:.
\end{equation*}
Then pluggining (\ref{counter_eq6}) and (\ref{counter_eq5}) in (\ref{counter_eq4}), we obtain 
\begin{equation} \label{counter_eq7}
\varphi_\lambda(s)= c {\left(\frac{s^{n-1}}{A(s)}\right)}^{1/2} \frac{1}{{(\lambda s)}^{\frac{n-1}{2}}} \left\{e^{i\left(\lambda s - \frac{\pi}{4}(n-1)\right)} + e^{-i\left(\lambda s - \frac{\pi}{4}(n-1)\right)}\right\} + \mathscr{E}_3(\lambda,s)\:,
\end{equation}
where
\begin{equation*}
\mathscr{E}_3(\lambda,s) = E_1(\lambda,s) + c {\left(\frac{s^{n-1}}{A(s)}\right)}^{1/2} \:\frac{\tilde{E}_{\frac{n-2}{2}}(\lambda s)}{(\lambda s)^{\left(\frac{n-2}{2}\right)}}\:,
\end{equation*}
and thus
\begin{eqnarray} \label{counter_eq8}
\left|\mathscr{E}_3(\lambda,s)\right| & \lesssim &  \left\{s^2 {(\lambda s)}^{-\frac{n+1}{2}} + {(\lambda s)}^{-\frac{n+1}{2}} \right\} \nonumber \\
& \lesssim & {(\lambda s)}^{-\frac{n+1}{2}}\:.
\end{eqnarray}
Thus for $s \in A_{N,\:\varepsilon}$ and $\lambda \in Supp(\widehat{f_N})$, we invoke the decomposition of $\varphi_\lambda$ given by (\ref{counter_eq7}) to write
\begin{equation} \label{counter_eq9}
Tf_N(s)= U_1f_N(s) + U_2f_N(s) + U_3f_N(s)\:,
\end{equation}
where
\begin{eqnarray*}
U_1f_N(s)&=& c \int_{N-N^{\left(1-\frac{a}{2}\right)}}^{N+N^{\left(1-\frac{a}{2}\right)}} {\left(\frac{s^{n-1}}{A(s)}\right)}^{1/2} \frac{1}{{(\lambda s)}^{\frac{n-1}{2}}} \: e^{i\left(\lambda s - \frac{\pi}{4}(n-1)\right)}  \:e^{it(s)\lambda^a}\:\widehat{f_N}(\lambda)\: {|{\bf c}(\lambda)|}^{-2}\: d\lambda\:, \\
U_2f_N(s)&=& c \int_{N-N^{\left(1-\frac{a}{2}\right)}}^{N+N^{\left(1-\frac{a}{2}\right)}} {\left(\frac{s^{n-1}}{A(s)}\right)}^{1/2} \frac{1}{{(\lambda s)}^{\frac{n-1}{2}}} \: e^{-i\left(\lambda s - \frac{\pi}{4}(n-1)\right)}  \:e^{it(s)\lambda^a}\:\widehat{f_N}(\lambda)\: {|{\bf c}(\lambda)|}^{-2}\: d\lambda\:,\\
U_3f_N(s)&=& \int_{N-N^{\left(1-\frac{a}{2}\right)}}^{N+N^{\left(1-\frac{a}{2}\right)}} \mathscr{E}_3(\lambda,s) \:e^{it(s)\lambda^a}\:\widehat{f_N}(\lambda)\: {|{\bf c}(\lambda)|}^{-2}\: d\lambda \:.
\end{eqnarray*}
By plugging in the definition of $\widehat{f_N}$ in terms of $\eta$ and then performing the change of variable $\xi=-N^{\left(\frac{a}{2}-1\right)} \lambda + N^{\frac{a}{2}}$, we get
\begin{eqnarray*}
&&U_2f_N(s)\\
&=& \frac{c\:e^{i\frac{\pi}{4}(n-1)}}{{A(s)}^{\frac{1}{2}}} N^{-\frac{1}{2}} \int_{N-N^{\left(1-\frac{a}{2}\right)}}^{N+N^{\left(1-\frac{a}{2}\right)}} e^{i\left\{-\lambda s +t(s)\lambda^a \right\}}  \:\eta \left(-N^{\left(\frac{a}{2}-1\right)} \lambda + N^{\frac{a}{2}}\right)\: \frac{{|{\bf c}(\lambda)|}^{-1}}{\lambda^{\frac{n-1}{2}}}\: d\lambda \\
&=& \frac{c\:e^{i\frac{\pi}{4}(n-1)}}{{A(s)}^{\frac{1}{2}}}N^{\frac{1}{2}(1-a)} \bigintssss_{-1}^1 e^{i \theta(\xi)}\:\zeta(\xi)\:d\xi\:,
\end{eqnarray*}
where 
\begin{eqnarray*}
&& \theta(\xi)= \left(N^{\left(1-\frac{a}{2}\right)} \xi - N\right)s + t(s) \left(-N^{\left(1-\frac{a}{2}\right)} \xi + N\right)^a \:, \\
&& \zeta(\xi)= \eta(\xi)\: \frac{{\left|{\bf c}\left(-N^{\left(1-\frac{a}{2}\right)} \xi + N\right)\right|}^{-1}}{\left(-N^{\left(1-\frac{a}{2}\right)} \xi + N\right)^{\frac{n-1}{2}}}\:.
\end{eqnarray*}
Now for $|\xi|<1$, by the generalized Binomial expansion, we get
\begin{equation*}
\theta(\xi)=\left(t(s)N^a -Ns\right)+\left(N^{\left(1-\frac{a}{2}\right)}s-a\:t(s)N^{\frac{a}{2}}\right)\xi +\frac{a(a-1)}{2}t(s)\xi^2 +\mathcal{O}\left(t(s)N^{-\frac{a}{2}}\right)\:.
\end{equation*}
Now for $s \in A_{N,\:\varepsilon}$, letting
\begin{equation*}
t(s)=\frac{sN^{1-a}}{a}\:,
\end{equation*}
we get rid of the linear term in $\xi$ to obtain,
\begin{equation*}
\theta(\xi)= Ns\left(\frac{1}{a}-1\right)+\frac{(a-1)}{2}sN^{1-a}\xi^2 +\mathcal{O}(1)\:.
\end{equation*}
Thus 
\begin{equation} \label{counter_eq10}
\left|U_2f_N(s)\right| \gtrsim A(s)^{-\frac{1}{2}}\:N^{\frac{1}{2}(1-a)}\left|\int_{-1}^1 e^{i\frac{(a-1)}{2}s'\xi^2} \:\eta(\xi)\: \frac{{\left|{\bf c}\left(-N^{\left(1-\frac{a}{2}\right)} \xi + N\right)\right|}^{-1}}{\left(-N^{\left(1-\frac{a}{2}\right)} \xi + N\right)^{\frac{n-1}{2}}}\:d\xi\right|\:,
\end{equation}
where we have made the substitution $s'=sN^{1-a}$ above. Then as $s \in A_{N,\:\varepsilon}$, it follows that $s' \in (a\varepsilon,2a\varepsilon)$\:. We further note that for $N \ge 2$ and $|\xi|<1$,
\begin{equation*}
\frac{{\left|{\bf c}\left(-N^{\left(1-\frac{a}{2}\right)} \xi + N\right)\right|}^{-1}}{\left(-N^{\left(1-\frac{a}{2}\right)} \xi + N\right)^{\frac{n-1}{2}}} \asymp 1\:.
\end{equation*}
Then expanding the integral appearing in the right hand side of (\ref{counter_eq10}) in terms of sine and cosine, using the triangle inequality, the above information, properties of $\eta$ and properties of sine and cosine, we get that for sufficiently small $\varepsilon \in \left(0,1/2\right)$, there exists $c'>0$, independent of $N$ such that
\begin{equation}\label{counter_estimate1}
\left|U_2f_N(s)\right|> c' \:A(s)^{-\frac{1}{2}}\:N^{\frac{1}{2}(1-a)}\:.
\end{equation}

We next show that for large $N$, the growths of both $U_1f_N$ and $U_3f_N$ are subordinated by $U_2f_N$. We first focus on $U_3f_N$. For $s \in A_{N,\:\varepsilon}$, using the local growth asymptotics of the density function (\ref{density_function}), the large frequency asymptotics of ${|{\bf c}(\cdot)|}^{-2}$ given by (\ref{plancherel_measure}), the remainder term estimates (\ref{counter_eq8}) and plugging in the definition of $\what{f_N}$ in terms of $\eta$, we get that there exists an integer $N_2 \ge 2$ such that for all $N \ge N_2$,
\begin{eqnarray} \label{counter_estimate2}
\left|U_3f_N(s)\right| &\le & c\:s^{-\frac{n+1}{2}}\:N^{-\frac{1}{2}}\int_{N-N^{\left(1-\frac{a}{2}\right)}}^{N+N^{\left(1-\frac{a}{2}\right)}}\lambda^{-\frac{n+1}{2}}\: {|{\bf c}(\lambda)|}^{-1}\: d\lambda \nonumber\\
&\le & c\: A(s)^{-\frac{1}{2}}\:s^{-1}\:N^{-\frac{1}{2}} \int_{N-N^{\left(1-\frac{a}{2}\right)}}^{N+N^{\left(1-\frac{a}{2}\right)}} \frac{d\lambda}{\lambda} \nonumber\\
&\le & \frac{c}{a\varepsilon}\: A(s)^{-\frac{1}{2}}\:N^{\left(\frac{1}{2}-a\right)} \log\left(\frac{N+N^{\left(1-\frac{a}{2}\right)}}{N-N^{\left(1-\frac{a}{2}\right)}}\right)\nonumber\\
&<& c''\:A(s)^{-\frac{1}{2}}\:N^{\frac{1}{2}(1-3a)}\:,
\end{eqnarray}
for some constant $c''>0$, independent of $N$. The inequality in the fourth line follows from the elementary estimate
\begin{equation*}
\log\left(\frac{x+1}{x-1}\right) \lesssim \frac{1}{x}\:,
\end{equation*}
which is valid for all large positive $x$.

We now turn to estimate $U_1f_N$. Again proceeding as in the case of $U_2f_N$, we get
\begin{equation*}
U_1f_N(s) = \frac{c\:e^{-i\frac{\pi}{4}(n-1)}}{{A(s)}^{\frac{1}{2}}}N^{\frac{1}{2}(1-a)} \bigintssss_{-1}^1 e^{i \theta(\xi)}\:\zeta(\xi)\:d\xi\:,
\end{equation*}
where 
\begin{eqnarray*}
&& \theta(\xi)= Ns- N^{\left(1-\frac{a}{2}\right)}s \xi + t(s) \left(-N^{\left(1-\frac{a}{2}\right)} \xi + N\right)^a \:, \\
&& \zeta(\xi)= \eta(\xi)\: \frac{{\left|{\bf c}\left(-N^{\left(1-\frac{a}{2}\right)} \xi + N\right)\right|}^{-1}}{\left(-N^{\left(1-\frac{a}{2}\right)} \xi + N\right)^{\frac{n-1}{2}}}\:.
\end{eqnarray*}
Now integration-by-parts yield,
\begin{equation} \label{counter_eq11}
\int_{-1}^1 e^{i\theta(\xi)}\:\zeta(\xi)\:d\xi = i \int_{-1}^1 e^{i\theta(\xi)}\: \left[\frac{\zeta'(\xi)}{\theta'(\xi)}-\frac{\zeta(\xi)\theta''(\xi)}{{\left(\theta'(\xi)\right)}^2}\right]\: d\xi\:.
\end{equation}
The expression (\ref{counter_eq11}) prompts us to obtain pointwise estimates of derivatives of the functions $\theta$ and $\zeta$\:. In this regard, we first note that
\begin{equation*}
\theta'(\xi)= - N^{\left(1-\frac{a}{2}\right)}\left[s+t(s)a\left(-N^{\left(1-\frac{a}{2}\right)} \xi + N\right)^{a-1}\right]
\end{equation*}  
Thus for $|\xi|<1$ and $s \in A_{N,\:\varepsilon}$,
\begin{equation} \label{counter_eq12}
|\theta'(\xi)| \ge N^{\left(1-\frac{a}{2}\right)}s>a\varepsilon N^\frac{a}{2}\:.
\end{equation}
We next note that
\begin{equation*}
\theta''(\xi)= t(s)a(a-1)\left(-N^{\left(1-\frac{a}{2}\right)} \xi + N\right)^{a-2}N^{2-a}\:.
\end{equation*}
Recalling that $|\xi|<1$, $t(s)=(sN^{1-a})/a$ and $s \in A_{N,\:\varepsilon}$, we get that there exists an integer $N_3 \ge 2$ such that for all $N \ge N_3$,
\begin{equation} \label{counter_eq13}
|\theta''(\xi)| \le c N^{a-2}N^{2-a}=c\:.
\end{equation}
Now recalling the definition of $\zeta$, we note that as $N\ge 2$ and $|\xi|<1$, 
\begin{equation*}
\frac{{\left|{\bf c}\left(-N^{\left(1-\frac{a}{2}\right)} \xi + N\right)\right|}^{-1}}{\left(-N^{\left(1-\frac{a}{2}\right)} \xi + N\right)^{\frac{n-1}{2}}} \asymp 1\:.
\end{equation*}
Hence
\begin{equation} \label{counter_eq14}
|\zeta(\xi)| \le c\:.
\end{equation}
We next note that
\begin{eqnarray*}
\zeta'(\xi) &=& \eta'(\xi)\:\frac{{\left|{\bf c}\left(-N^{\left(1-\frac{a}{2}\right)} \xi + N\right)\right|}^{-1}}{\left(-N^{\left(1-\frac{a}{2}\right)} \xi + N\right)^{\frac{n-1}{2}}}\: +\: \eta(\xi)\:\frac{\frac{d}{d\xi}\left({\left|{\bf c}\left(-N^{\left(1-\frac{a}{2}\right)} \xi + N\right)\right|}^{-1}\right)}{\left(-N^{\left(1-\frac{a}{2}\right)} \xi + N\right)^{\frac{n-1}{2}}}  \\
&&- \eta(\xi)\:{\left|{\bf c}\left(-N^{\left(1-\frac{a}{2}\right)} \xi + N\right)\right|}^{-1}\:\frac{\frac{d}{d\xi}\left(\left(-N^{\left(1-\frac{a}{2}\right)} \xi + N\right)^{\frac{n-1}{2}}\right)}{\left(-N^{\left(1-\frac{a}{2}\right)} \xi + N\right)^{n-1}}\:.
\end{eqnarray*}
Then using the pointwise and derivative estimates of ${\left|\bf c(\cdot)\right|}^{-2}$ given in (\ref{plancherel_measure}) and (\ref{c-fn_derivative_estimates}), we get that an integer $N_4 \ge 2$, such that for all $N \ge N_4$,
\begin{equation} \label{counter_eq15}
|\zeta'(\xi)| \le c \left(1+ N^{-\frac{a}{2}} +  N^{-\frac{a}{2}}\right) \le c\:.
\end{equation}
Thus for $s \in A_{N,\:\varepsilon}$, using (\ref{counter_eq11})-(\ref{counter_eq15}), there exists an integer $N_5 \ge 2$, such that for all $N\ge N_5$,
\begin{eqnarray} \label{counter_estimate3}
\left|U_1f_N(s)\right| &\le & c\:A(s)^{-\frac{1}{2}}\:N^{\frac{1}{2}(1-a)} \left[N^{-\frac{a}{2}}\:+\:N^{-a} \right] \nonumber \\
& < & c'''\:\:A(s)^{-\frac{1}{2}}\:N^{\frac{1}{2}(1-2a)}  \:,
\end{eqnarray}
for some positive constant $c'''$ independent of $N$. 

Hence setting $N_0=\displaystyle\max_{1 \le j \le 5}N_j$, for all $N>N_0$ and $s \in A_{N,\:\varepsilon}$, in view of the decomposition (\ref{counter_eq9}) we get from (\ref{counter_estimate1}), (\ref{counter_estimate2}) and (\ref{counter_estimate3}) that
\begin{equation*}
\left|Tf_N(s)\right| \ge  \left|U_2f_N(s)\right| - \left|U_1f_N(s)\right| - \left|U_3f_N(s)\right| > c_4\:A(s)^{-\frac{1}{2}}\:N^{\frac{1}{2}(1-a)}\:,
\end{equation*}
for some positive constant $c_4$, independent of $N$\:. This then implies that for $N>N_0$,
\begin{equation*}
\|Tf_N\|_{L^2(B_1)} > c_4\left(\int_{a\varepsilon N^{a-1}}^{2a\varepsilon N^{a-1}} N^{1-a}\:ds\right)^{\frac{1}{2}}=c_4 (a\varepsilon)^\frac{1}{2}\:.
\end{equation*}
Then setting $c_3=c_4 (a\varepsilon)^\frac{1}{2}$, we get (\ref{counter_eq3}) and hence the proof of Theorem \ref{sharpness_thm} for the Fractional Schr\"odinger equation of degree $a$, corresponding to $\tilde{\Delta}$, is completed. 

Now, the general case of the asymptotically concave dispersive equations of degree $a$ follows from the result for the fractional Schr\"odinger equation which we just proved and the transference principle (Lemma \ref{transference_principle}). Indeed, in their case, if the estimate (\ref{maximal_bddness_inequality}) is true for some $\beta_0<a/4$, then it would also be true for any $\beta>\beta_0$. Then by the transference principle Lemma \ref{transference_principle} and Remark \ref{concavity_transference}, (\ref{maximal_bddness_inequality}) is also true for the fractional Schr\"odinger equation of degree $a$, for any $\beta>\beta_0$ and hence in particular for the choice 
\begin{equation*}
\beta = \frac{1}{2}\left(\beta_0 + \frac{a}{4}\right) < \frac{a}{4}\:.
\end{equation*} 
But that is a contradiction. This completes the proof of Theorem \ref{sharpness_thm}.

\section{Concluding remarks}
In this section, we make some remarks and pose some new problems:
\begin{enumerate}
\item Exact analogues of Theorems \ref{maximal_bddness_thm} and \ref{sharpness_thm} can be obtained for $\R^n$. This will generalize Walther's results in \cite{Wa2,Wa3} for the fractional Schr\"odinger equations of degree $a \in (0,1)$, to the general class of {\em asymptotically concave dispersive equations of degree} $a \in (0,1)$. The key here is to obtain an analogue of the transference principle (Lemma \ref{transference_principle}) in the setting of $\R^n$. 
\item Application of the transference principle (Lemma \ref{transference_principle}) crucially uses the concavity of the phase. This approach does not work in the case when $a>2$ (see \cite[point (i) of Remark $1.5$]{Dewan2}).
\item A natural question is to study the behaviour at the end-point $\beta=a/4$. In fact, this problem is still open  in $\R^n$ itself.  
\end{enumerate}

\section*{Acknowledgements} 
The author is supported by a research fellowship of Indian Statistical Institute.

\bibliographystyle{amsplain}

\end{document}